\documentclass{article}
\usepackage[utf8]{inputenc}
\usepackage[left=1in,top=1in,right=1in,bottom=1in]{geometry}
\usepackage[linktocpage=true]{hyperref}
\usepackage{setspace}

\usepackage{tikz}
\usetikzlibrary{decorations.pathreplacing}
\usetikzlibrary{calc, arrows.meta}
\tikzset{vtx/.style={circle, fill, inner sep=1.5pt}}
\tikzset{openvtx/.style={circle, draw, inner sep=1.5pt}}

\usepackage{amssymb, amsmath, amsthm, graphicx,mathrsfs,xcolor}
\usepackage{caption}
\usepackage[noadjust]{cite}
\usepackage{subcaption}
\usepackage{mathtools,color}
\usepackage{tabularx}
\usepackage{xltabular}
\usepackage[capitalise]{cleveref}
\crefname{obs}{Observation}{Observations}

\usepackage{todonotes}

\usepackage{enumitem}
\newlist{defitemize}{itemize}{1}
\setlist[defitemize,1]{
leftmargin=20pt, label=$\bullet$\hspace*{5pt},
parsep=0pt, partopsep=0pt,
itemsep=3pt, topsep=5pt,
}

\newlist{myproofsteps}{enumerate}{1}
\setlist[myproofsteps,1] {
label=Step \arabic*.,
parsep=0pt, partopsep=0pt,
itemsep=\medskipamount, topsep=5pt,
align=left,left=0pt,
labelsep=-0.15in,
}
\newenvironment{proofsteps}{
\begin{myproofsteps}
\setlength{\parindent}{0.25in}
\newcommand{\step}[1]{\item[Step ##1.]\hspace*{0.15in}}
}
{\end{myproofsteps}}

\newlist{myalgsteps}{enumerate}{1}
\setlist[myalgsteps,1] {
label=Step \arabic*.,
parsep=0pt, partopsep=0pt,
itemsep=2pt, topsep=5pt,
align=left,left=0pt,
}
\newenvironment{algsteps}{
\begin{myalgsteps}
\newcommand{\step}[1]{\item[Step ##1.]}
}
{\end{myalgsteps}}

\numberwithin{equation}{section}
\newtheorem{thm}{Theorem}[section]

\newtheorem{lem}[thm]{Lemma}
\newtheorem{quest}[thm]{Question}

\newtheorem{conj}[thm]{Conjecture}
\newtheorem{prop}[thm]{Proposition}

\newtheorem{obs}[thm]{Observation}

\theoremstyle{definition}
\newtheorem{defn}[thm]{Definition}

\newtheorem{rem}[thm]{Remark}

\definecolor{cbpink}{HTML}{DC267F}

\newcommand{\Z}{\mathbb{Z}}

\newcommand{\A}{\mathcal{A}}
\newcommand{\B}{\mathcal{B}}

\newcommand{\ex}{\mathrm{ex}}

\newcommand{\dist}{\mathrm{dist}}

\newcommand{\sm}{\setminus}
\newcommand{\sub}{\subseteq}
\newcommand{\del}{\delta}

\renewcommand{\c}[1]{\mathcal{#1}}

\newcommand{\lfl}{\lfloor}
\newcommand{\rfl}{\rfloor}
\newcommand{\lcl}{\lceil}
\newcommand{\rcl}{\rceil}

\DeclarePairedDelimiter{\abs}{\lvert}{\rvert}
\DeclarePairedDelimiter{\floor}{\lfloor}{\rfloor}
\DeclarePairedDelimiter{\ceil}{\lceil}{\rceil}

\newcommand{\zbad}{z_{\mathrm{bad}}}
\newcommand{\coox}{\vec q}

\title{Rainbow Trees in Hypercubes}
\author{ Nicholas Crawford\thanks{Los Medanos {\texttt{ncrawford@losmedanos.edu}}} \and Maya Sankar\thanks{Stanford University {\texttt{mayars@stanford.edu}}} \and Carl Schildkraut\thanks{Stanford University {\texttt {carlsch@stanford.edu}}}\and Sam Spiro\thanks{Georgia State University {\texttt {sspiro@gsu.edu}}.}}
\date{\today}
\begin{document}

\maketitle

\begin{abstract}
We prove that every proper edge-coloring of the $n$-dimensional hypercube $Q_n$ contains a rainbow copy of every tree $T$ on at most $n$ edges.  This result is best possible, as $Q_n$ can be properly edge-colored using only $n$ colors while avoiding rainbow cycles.
\end{abstract}

\section{Introduction}

One of the most well-studied graphs is the \textit{$n$-dimensional hypercube} $Q_n$, whose vertex set consists of elements of $\mathbb{F}_2^n$, with two vertices adjacent if and only if they differ in exactly one coordinate.

Motivated by problems related to Schrijver's Conjecture and the rainbow Tur\'an problem (which we define in greater detail below), the first author, King, and the fourth author \cite{cks-25} initiated the study of which rainbow subgraphs can appear in a properly edge-colored subgraphs of the hypercube. Specifically, the simplest case of the main conjectures made in \cite{cks-25} is the following, where here a \textit{rainbow} subgraph of an edge-colored graph $G$ is a subgraph of $G$ where every edge has a distinct color.

\begin{conj}[{\cite{cks-25}}]\label{main conjecture}
Every proper edge-coloring of the $n$-dimensional hypercube contains a rainbow copy of every tree $T$ with at most $n$ edges.
\end{conj}

\cref{main conjecture} is best possible, in the sense that if $H$ is not a tree with at most $n$ edges, then $Q_n$ possesses an edge-coloring without a rainbow copy of $H$. Indeed, consider the Cayley edge-coloring of $Q_n$, i.e., the coloring where an edge $uv$ whose vertices differ on their $i$th coordinates is given color $i$. This coloring has only $n$ colors and no rainbow cycles.

\Cref{main conjecture} was proven for a few classes of trees in \cite{cks-25}, specifically for trees which either contain many leaves or which contain a long pendant path.  In this article, we give a full resolution of \Cref{main conjecture} by proving the following stronger statement, which appeared as \cite[Conjecture~3]{cks-25}. Here and throughout, $\delta(G)$ denotes the minimum degree of the graph $G$.

\begin{thm}\label{main theorem}
Let $G$ be a properly edge-colored subgraph of a hypercube $Q_n$. Then, for every tree $T$ with at most $\delta(G)$ edges, there exists a rainbow copy of $T$ in $G$.
\end{thm}

Our proof of \cref{main theorem} proceeds by a delicate induction argument, resulting in a more general (albeit more technical) statement \cref{patch-main-thm}. Although the precise statement of \cref{patch-main-thm} is too involved to state here, it informally says that any sufficiently ``well-dispersed'' embedding of the first half of $E(T)$ can be extended to a rainbow embedding of all of $T$. 
Such an embedding of the first half of $E(T)$ can be constructed relatively easily, as discussed in \cref{remark:dd-embedding}.

\subsection{Related Results and Historical Context}

A growing interest within extremal combinatorics is to study sufficient conditions for combinatorial objects to contain ``rainbow'' substructures.  Here we highlight a few problems around the theme of finding rainbow subgraphs in properly edge-colored graphs, and we refer the interested reader to the surveys \cite{NP-22,s-24} for a more complete overview of this topic.

Perhaps the most famous problem of this form is the Ryser--Brualdi--Stein conjecture.  The usual phrasing of this conjecture states that every Latin square of order $n$ contains a transversal of size $n-1$, with the existence of a transversal of size $n$ being guaranteed whenever $n$ is odd.  Equivalently, this conjecture says that any proper edge-coloring of $K_{n,n}$  using $n$ colors contains a rainbow matching of size $n-1$, with a rainbow matching of size $n$ existing whenever $n$ is odd.  A large number of partial results have been obtained for this problem, such as work of Shor \cite{s-82} in 1982 showing the existence of transversals of size $n-O(\log^2 n)$, though this work contained an error which was later corrected by Hatami and Shor in 2008 \cite{hs-08}. This bound stood for many years until recent work of Keevash, Pokrovskiy, Sudakov, and Yepremyan  \cite{kpsy-20} who in 2020 proved the existence of transversals of size $n-O(\log n/\log\log n)$.  Very recently, breakthrough work of Montgomery \cite{m-23} verified the Ryser--Brualdi--Stein conjecture for all large even $n$. We refer readers to \cite{m-23} for the most up-to-date information on the Ryser--Brualdi--Stein conjecture.

Another problem in a similar spirit is Graham's rearrangement conjecture from 1971, which says that for any prime $p$ and $a_1,\ldots,a_d$ non-zero distinct elements of $\Z/p\Z$, there exists a rearrangement of the elements $a_{i_1},\ldots,a_{i_d}$ such that the partial sums $\sum_{j=1}^t a_{i_j}$ are all distinct.  By translating this problem appropriately, one can show that this conjecture is equivalent to showing that a certain class of edge-colored directed graphs, or \emph{digraphs}, contain long rainbow directed paths. In particular, Graham's rearrangement conjecture, as well as various strengthenings thereof, would be implied by the following question asked very recently in 2025 by Buci\'c, Frederickson, M\"uyesser, Pokrovskiy, and Yepremyan \cite{bfmpy-25}; here a proper edge-coloring of a digraph is one in which no two directed edges sharing either a common source or a common sink are given the same color.
\begin{quest}[{\cite{bfmpy-25}}]\label{directed Schrijver}
If $D$ is a properly edge-colored digraph such that every vertex has both in-degree and out-degree equal to $d$, does $D$ contain a rainbow directed path of length $d-1$? 
\end{quest}
This question in turn is a significant strengthening of the following conjecture of Schrijver~\cite{s-18} which itself is a strengthening of a conjecture of Andersen \cite{a-89} who made the same statement only for the case $G=K_n$.
\begin{conj}[Schrijver~\cite{s-18}]
If $G$ is a properly edge-colored $d$-regular graph, then $G$ contains a rainbow path of length $d-1$. 
\end{conj}
In both these problems the quantity $d-1$ is best possible due to a construction of Maamoun and Meyniel \cite{mm-84}, the simplest version of which starts by taking the hypercube $Q_{d-1}$, equipped with the Cayley edge-coloring described above, and then adding all of the diagonals of $Q_{d-1}$ as edges in a $d$th color. Schrijver's conjecture has been asymptotically resolved by Buci\'c, Frederickson, M\"uyesser, Pokrovskiy, and Yepremyan \cite{bfmpy-25}, and we refer the interested reader to their paper for more on Graham's rearrangement conjecture and on problems related to finding rainbow paths in graphs and digraphs.

The last related problem we discuss in full is the rainbow Tur\'an problem, which was originally initiated by Keevash, Mubayi, Sudakov, and Verstra\"ete \cite{kmsv-07} in 2007. Given a family of graphs $\mathcal{F}$, let $\ex^*(n,\mathcal{F})$ denote the maximum number of edges in an $n$-vertex graph $G$ admitting a proper edge-coloring containing no element of $\mathcal F$ as a rainbow subgraph. For example, the hypercube $Q_d$ together with its Cayley edge-coloring contains no rainbow cycle, so taking $d\approx \log n$ implies $\ex^*(n,\mathcal{C})=\Omega(n \log n)$ where $\mathcal{C}$ denotes the set of all cycles.  Determining a corresponding upper bound for $\ex^*(n,\mathcal{C})$ has been a significant area of study and, after a number of preceding works \cite{dls-12, t-24, j-21}, was largely solved in 2025 by Alon, Buci\'c, Sauermann, Zakharov, and Zamir \cite{absz-25} who showed both that $\ex^*(n,\mathcal{C})\le n (\log n )^{1+o(1)}$ as well as some interesting connections between this problem and additive number theory.

Our main problem of interest for this paper, \Cref{main conjecture} arose from trying to find a rainbow analog of the Erd\H{o}s--S\'os Conjecture from extremal graph theory, which states that every tree $T$ on $d$ edges has the same Tur\'an number as the star on $d$ edges.  The most natural analog of this conjecture, namely that $\ex^*(n,T)=\ex^*(n,K_{1,d})$ for every tree $T$ on $d$ edges, is false essentially due to the same construction of Maamoun and Meyniel \cite{mm-84} for Schrijver's conjecture mentioned above.  To circumvent this, the authors in \cite{cks-25} considered a relative rainbow Tur\'an number $\ex^*(G,T)$, which is defined to be the maximum number of edges in a subgraph $G'\sub G$ which can be properly edge-colored to avoid rainbow copies of $T$.  It was then asked in \cite{cks-25} which host graphs $G$ have the property that $\ex^*(G,T)=\ex^*(G,K_{1,d})$ for all $d$-edge trees.  In particular, it was conjectured that the hypercube might have this property, with the motivation for this particular choice of host coming both from the fact that the Maamoun--Meyniel construction shows that graphs slightly denser than the hypercube can fail to have this property, as well as from the canonical coloring of the hypercube showing that $\ex^*(Q_d,F)$ is only interesting to study when $F$ does not contain a cycle.  Our main result \Cref{main theorem} can thus be seen as a minimum degree analog of the conjecture that $\ex^*(Q_n,T)=\ex^*(Q_n,K_{1,d})$ for all $d$-edge trees $T$.

\subsection{Paper Organization}
The rest of our paper is organized as follows.  In \Cref{sec:prelim} we outline some of the main proof ideas and establish several key lemmas and definitions.  The bulk of our proof is then given in \Cref{sec:main} where we prove our main technical result, \Cref{patch-main-thm}.  We close with some open problems and related results in \Cref{sec:con}.

\section{Preliminaries}\label{sec:prelim}
In this section we introduce some of the necessary definitions and basic lemmas for our main arguments.  For convenience, we split our discussion into three parts: one focused on the high-level proof ideas and basic observations that we use, one on properties of trees, and one on embeddings into the hypercube.

\subsection{Proof Ideas}
In this subsection we discuss some of the high-level proof ideas of our argument.  For this, the following basic definitions will be useful.

\begin{defn}[Infinite hypercube]\label{def:hypercube} The (infinite) \emph{hypercube} $Q_\infty$ is the graph whose vertex set is $\{0,1\}^{\mathbb N}$ and whose edge set is given by connecting two vertices $x,y$ if they differ in exactly one coordinate. Given an edge $e:=xy\in E(Q_\infty)$, the \emph{coordinate} of $e$ is the unique $i\in\mathbb N$ at which $x$ and $y$ differ.
\end{defn}

Throughout this paper we are solely concerned with finite subgraphs of $Q_\infty$, and as such we could alternatively consider each of our graphs as subgraphs of some finite hypercube $Q_N$ (with vertex set $\mathbb{F}_2^N$) rather than $Q_\infty$. We use $Q_\infty$ throughout to emphasize that the dimensionality of the ambient hypercube in which we are working has no importance.

\begin{defn}[Homomorphisms and embeddings]
Given two graphs $F,G$, a map $\phi\colon V(F)\to V(G)$ is a \textit{homomorphism} if $\phi$ maps edges of $F$ to edges of $G$; that is, if for every edge $xy\in E(F)$ we have $\phi(x)\phi(y)\in E(G)$.  A homomorphism $\phi$ is called an \textit{embedding} if it is injective.

We frequently abuse notation by writing $\phi(xy)$ for an edge $xy\in E(F)$ to denote the edge $\phi(x)\phi(y)$ of $G$. Similarly, given a set of edges $E\sub E(F)$ we write $\phi(E):=\{\phi(e):e\in E\}$.  We often write $\phi\colon F\to G$ rather than the more precise $\phi\colon V(F)\to V(G)$ for ease of notation.  
\end{defn}

Ultimately we aim to show that if $G\sub Q_\infty$ is a properly edge-colored graph and $T$ is a tree with $e(T)\ge \delta(G)$, then there exists a homomorphism $\phi\colon T\to G$ which is both an embedding and has the property that distinct edges of $T$ map to edges of $G$ with distinct colors.  The distinct colors condition alone is easy to ensure by iteratively constructing $\phi$ to avoid previously used colors, so in some sense our main challenge will be defining $\phi$ so that distinct vertices of $T$ map to distinct vertices of $G$.  We will do this via the following key lemma regarding the geometry of the hypercube.

\begin{obs}[Non-closedness of walks in $Q_\infty$]\label{fact:distinct} 
Let $x_0x_1\cdots x_m$ be a walk in $Q_\infty$. If the number of distinct coordinates used by the $m$ edges $x_0x_1,\ldots,x_{m-1}x_m$ is strictly more than $m/2$, then $x_0\neq x_m$.
\end{obs}
\begin{proof} By the pigeonhole principle, there must be some coordinate $j$ which is the coordinate of exactly one edge $x_{i-1}x_i$ with $1\leq i\leq m$. This implies that
\[(x_0)_j=\cdots=(x_{i-1})_j\neq (x_i)_j=\cdots=(x_m)_j,\]
proving that $x_0$ and $x_m$ are distinct vertices.
\end{proof}

Hence, to ensure that two vertices $u,v$ of $T$ map to distinct vertices of $Q_\infty$ under a given homomorphism $\phi\colon T\to Q_\infty$, it suffices to show that the edge set of the unique path from $u$ to $v$ in $T$ is mapped to a set of edges of $Q_{\infty}$ encompassing strictly more than $\dist_T(u,v)/2$ distinct coordinates.

Motivated by the above, we aim to iteratively construct the homomorphism $\phi\colon T\to G$ in a manner that never repeats colors (which is necessary for the image of $\phi$ to be rainbow) and in a manner that avoids repeating coordinates as much as possible (with the intent to apply \cref{fact:distinct} to prove that $\phi$ is an embedding). The following observation quantifies how much we can repeat avoiding coordinates. 

\newcommand{\xcol}{X_{\mathrm{col}}}
\newcommand{\xcor}{X_{\mathrm{coor}}}

\begin{obs}[Extending a rainbow homomorphism]\label{fact:extendr}
Let $G$ be a properly edge-colored subgraph of $Q_\infty$ and fix a vertex $x\in V(G)$. Let $\xcol$ be a finite subset of the colors used by edges of $G$ and $\xcor\subseteq\mathbb N$ a finite set of coordinates. Suppose that there are at least $r$ neighbors $x'$ of $x$ in $G$ such that the edge $xx'$ has color in $\xcol$ and coordinate in $\xcor$. If $\abs{\xcor}+\abs{\xcol}-r<\delta(G)$ then there is an edge $xy\in E(G)$ with color not in $\xcol$ and coordinate not in $\xcor$.
\end{obs}

We  always apply \cref{fact:extendr} in the context of extending a homomorphism $\phi'\colon (T-w)\to G$ to a homomorphism $\phi\colon T\to G$, where $w$ is a leaf of a tree $T$. Writing $v\in V(T)$ for the sole neighbor of $w$, one should interpret the vertex $x$ in \cref{fact:extendr} as the image $\phi'(v)$ and the vertex $y$ as a potential choice for $\phi(w)$ so that the edge $\phi(vw)$ avoids a set of forbidden colors $\xcol$, representing the colors of edges used by $\phi(T-w)$, together with a set of forbidden coordinates $\xcor$.

\begin{proof}
    Let $N(x)$ be the set of neighbors of $x$ in $G$.  Further, let $N_{\mathrm{col}}(x)$ denote the set of neighbors $x'$ with $xx'$ using a color in $\xcol$, and define  $N_{\mathrm{coor}}(x)$ similarly. We observe that
    \begin{align*}\abs{N(x)\setminus (N_{\mathrm{col}}(x)\cup N_{\mathrm{coor}}(x))}&=\abs{N(x)}-\abs{N_{\mathrm{col}}(x)}-\abs{N_{\mathrm{coor}}(x)}+\abs{N_{\mathrm{col}}(x)\cap N_{\mathrm{coor}}(x)}\\ &\ge \delta(G)-\abs{\xcor}-\abs{\xcol}+r\ge 1,\end{align*}
    where the first inequality used both that $G$ has a proper edge coloring to imply that at most one neighbor of $x$ uses any given color from $\xcol$, and also that $G\subseteq Q_\infty$ implies that at most one neighbor of $x$ uses any given coordinate from $\xcor$. The inequality above implies that there exists a neighbor $y$ of $x$ satisfying the desired properties.
\end{proof} 

\begin{rem}\label{remark:dd-embedding}
For trees $T$ and graphs $G$ with $e(T)=\delta(G)$, iterated applications of \cref{fact:extendr} with $r=0$ allows us to build a homomorphism which maps approximately half the edges of $T$ into $G$ while avoiding any repetition of colors or coordinates. 
\end{rem}

\cref{remark:dd-embedding} motivates the primary strategy for our proof: given a tree $T$ and an edge-colored graph $G\sub Q_{\infty}$ into which we aim to embed $T$, we first strategically pick some subtree $T'\sub T$ with about half as many edges as $T$ and then embed $T'$ into $G$ without any repetitions in coordinates or colors. From here, we order the remaining vertices of $T$ in some clever way and iteratively embed each vertex while avoiding all previously used colors as well as a carefully chosen set of previously used coordinates (which notably can be no larger than the set guaranteed in \Cref{fact:extendr}) at each step.  While this proof ends up being entirely elementary, we emphasize that the exact choices of the subtree $T'$, the ordering of the remaining vertices, and the set of forbidden coordinates to choose at each step are extremely delicate.

Before closing this subsection, we record the $r=1$ case of \Cref{fact:extendr} separately, as we make repeated use of this special case.

\begin{obs}[Extending a rainbow homomorphism, special case]\label{fact:extend1}
Let $G$ be a properly edge-colored subgraph of $Q_\infty$ and fix a vertex $x\in V(G)$. Let $\xcol$ be a finite subset of the colors used by edges of $G$ and $\xcor\subseteq\mathbb N$ a finite set of coordinates, and suppose there is an edge $xx'\in E(G)$ incident to $x$ whose color is in $\xcol$ and whose coordinate is in $\xcor$. If $\abs{\xcol}+\abs{\xcor}\leq\delta(G)$ then there is an edge $xy\in E(G)$ incident to $x$ with color not in $\xcol$ and coordinate not in $\xcor$.
\end{obs}

\subsection{Trees}

For the remainder of the paper we will consider all trees $T$ to be rooted, with the root vertex typically denoted as $v_0$ unless stated otherwise. We now establish some structural definitions of trees. 

\begin{defn}[Properties of trees]\label{def:tree} Given a tree $T$ with root $v_0$, we define the following notions.

\begin{defitemize}
    \item For a vertex $v\in V(T)$, the \emph{descendants} of $v$ are the vertices $w\in V(T)$ for which the unique path between $v_0$ and $w$ in $T$ contains $v$. 

    \item A \emph{leaf} is a vertex of $T$ with no descendants other than itself. Notably, $v_0$ is not considered to be a leaf unless $T$ consists of a single vertex.
    
    \item For a vertex $v\in V(T)$, we write $T(v)$ for the subtree of $T$ rooted at $v$ induced by its descendants (including $v$). We write $T-v$ for the subtree of $T$ rooted at $v_0$ induced by the vertex set $V(T) \setminus V(T(v))$.
    
    \item For a vertex $v\in V(T)$, the \emph{level} $\lambda(v)$ of $v$ is defined to be the distance between $v$ and $v_0$, i.e.\ $\lambda(v):=\dist_T(v,v_0)$. We define
    \[\lambda_{\max}(v):=\max_{\substack{\text{descendants}\\w\text{ of }v\text{ in }T}}\lambda(w)\]
    to be the greatest level among all of the descendants of $v$.
    
    \item Define $\floor{T/2}$ and $\ceil{T/2}$ to be the subtrees of $T$ rooted at $v_0$ induced by the vertex sets
    \begin{align*}
        V(\floor{T/2})&=\{v\in V(T):\lambda(v)\leq\lfl \lambda_{\max}(v)/2\rfl\} \text{ and}\\
        V(\ceil{T/2})&=\{v\in V(T):\lambda(v)\leq\lcl \lambda_{\max}(v)/2\rcl\}.
    \end{align*}
    In other words, $\floor{T/2}$ is the union of the first halves of all paths from $v_0$ to a leaf of $T$, excluding the middle edges of odd-length paths, and $\ceil{T/2}$ is the same with middle edges of odd-length paths included. We include an example of this in \Cref{fig:enter-label}. 
\end{defitemize}
\end{defn}

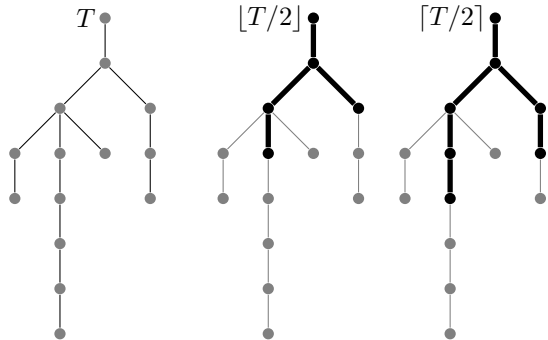
\begin{figure}[h]
\centering
\quad
\tikzset{exstyle1/.style={scale=0.6,baseline=-2cm,gvtx/.style={vtx,gray}}}
\begin{tikzpicture}[exstyle1,vtx/.append style={gray}]
    \path (0,0) coordinate[vtx] (x0) node[left] {$T$}
    --++ (0,-1) coordinate[vtx] (x1)
    --++ (-1,-1) coordinate[vtx] (x2)
    --++ (-1,-1) coordinate[vtx] (x3)
    --++ (0,-1) coordinate[vtx] (x4)
    (x2) --++ (0,-1) coordinate[vtx] (x5)
    --++ (0,-1) coordinate[vtx] (x6)
    --++ (0,-1) coordinate[vtx] (x7)
    --++ (0,-1) coordinate[vtx] (x12)
    --++ (0,-1) coordinate[vtx] (x13)
    (x2) --++ (1,-1) coordinate[vtx] (x8)
    (x1) --++ (1,-1) coordinate[vtx] (x9)
    --++ (0,-1) coordinate[vtx] (x10)
    --++ (0,-1) coordinate[vtx] (x11)
    ;
    \draw (x0) -- (x1) -- (x2) -- (x3) -- (x4)
    (x2) -- (x5) -- (x6) -- (x7) -- (x12) -- (x13)
    (x2) -- (x8)
    (x1) -- (x9) -- (x10) -- (x11)
    ;
\end{tikzpicture}
\qquad
\begin{tikzpicture}[exstyle1]
    \path (0,0) coordinate[vtx] (x0) node[left] {$\floor{T/2}$}
    --++ (0,-1) coordinate[vtx] (x1)
    --++ (-1,-1) coordinate[vtx] (x2)
    --++ (-1,-1) coordinate[gvtx] (x3)
    --++ (0,-1) coordinate[gvtx] (x4)
    (x2) --++ (0,-1) coordinate[vtx] (x5)
    --++ (0,-1) coordinate[gvtx] (x6)
    --++ (0,-1) coordinate[gvtx] (x7)
    --++ (0,-1) coordinate[gvtx] (x12)
    --++ (0,-1) coordinate[gvtx] (x13)
    (x2) --++ (1,-1) coordinate[gvtx] (x8)
    (x1) --++ (1,-1) coordinate[vtx] (x9)
    --++ (0,-1) coordinate[gvtx] (x10)
    --++ (0,-1) coordinate[gvtx] (x11)
    ;
    \draw[line width = 2pt] (x0) -- (x1) -- (x2)
    (x2) -- (x5)
    (x1) -- (x9)
    ;
    \draw[gray] (x2) -- (x3) -- (x4) 
    (x6) -- (x7) -- (x12) -- (x13)
    (x5) -- (x6)
    (x2) -- (x8) 
    (x9) -- (x10) -- (x11);
\end{tikzpicture}
\quad
\begin{tikzpicture}[exstyle1]
    \path (0,0) coordinate[vtx] (x0) node[left] {$\ceil{T/2}$}
    --++ (0,-1) coordinate[vtx] (x1)
    --++ (-1,-1) coordinate[vtx] (x2)
    --++ (-1,-1) coordinate[gvtx] (x3)
    --++ (0,-1) coordinate[gvtx] (x4)
    (x2) --++ (0,-1) coordinate[vtx] (x5)
    --++ (0,-1) coordinate[vtx] (x6)
    --++ (0,-1) coordinate[gvtx] (x7)
    --++ (0,-1) coordinate[gvtx] (x12)
    --++ (0,-1) coordinate[gvtx] (x13)
    (x2) --++ (1,-1) coordinate[gvtx] (x8)
    (x1) --++ (1,-1) coordinate[vtx] (x9)
    --++ (0,-1) coordinate[vtx] (x10)
    --++ (0,-1) coordinate[gvtx] (x11)
    ;
    \draw[line width = 2pt] (x0) -- (x1) -- (x2)
    (x2) -- (x5) -- (x6)
    (x1) -- (x9) -- (x10)
    ;
    \draw[gray] (x2) -- (x3) -- (x4) 
    (x6) -- (x7) -- (x12) -- (x13) 
    (x2) -- (x8) 
    (x10) -- (x11);
\end{tikzpicture}

\caption{An example of a tree $T$ and the corresponding trees $\floor{T/2}$ and $\ceil{T/2}$.}
\label{fig:enter-label}
\end{figure}

Ultimately, our embedding of $T$ will be constructed by starting with an embedding of $\floor{T/2}$.  Because of this, our embedding will be hardest to complete whenever $e(\floor{T/2})$ is large, as this means there will be more restrictions in the colors we need to avoid after embedding the edges in $\floor{T/2}$.  As such, it is useful to characterize which trees have $e(\floor{T/2})$ as large as possible.  Our characterization relies on spiders, which informally are trees formed by subdividing the edges of a star $K_{1,m}$.  More precisely we work with the following.

\begin{defn}[Spiders]\label{def:spider} A \emph{spider} is a rooted tree $T$ with root $v_0$ for which $\deg(v)\leq 2$ for every $v\in V(T)\setminus \{v_0\}$. For each leaf $w$ of $T$, the unique path from $v_0$ to $w$ is called a \emph{leg} of $T$; the legs partition the edge set of $T$. We say a spider is \emph{even} if all of its legs have even length. 
\end{defn}

\begin{figure}[h]
\centering
\begin{tikzpicture}[scale=1]
    
    \draw (0,0) coordinate[vtx] (s1) node[left] {$v_0$};
    \coordinate (child) at (0,-1);
    \begin{scope}[yscale=0.7]
    \draw (s1) -- (-1,0|-child) node[vtx]{}
        \foreach \i in {0,1,2,3} {--++ (0,-1) node[vtx] {}};
    \draw (s1) -- (-0.5,0|-child) node[vtx]{}
        \foreach \i in {0,1,2} {--++ (0,-1) node[vtx] {}};
    \draw (s1) -- (0,0|-child) node[vtx]{}
        \foreach \i in {0,1,2} {--++ (0,-1) node[vtx] {}};
    \draw (s1) -- (0.5,0|-child) node[vtx]{}
        \foreach \i in {0} {--++ (0,-1) node[vtx] {}};
    \draw (s1) -- (1,0|-child) node[vtx]{}
        \foreach \i in {0} {--++ (0,-1) node[vtx] {}};
    \end{scope}
\end{tikzpicture}
\caption{A spider rooted at $v_0$ with five legs: one of length 5, two of length 4, and two of length 2.}
\label{fig:enter-label}
\end{figure}
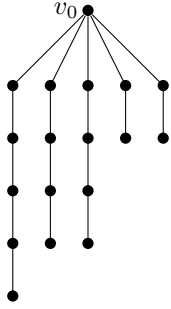

\begin{prop}\label{char tree attempt 2} Let $T$ be a rooted tree.
\begin{enumerate}[label=(\arabic*)]
    \item It holds that $e(T)\geq 2e(\floor{T/2})$.
    
    \item If $e(T)=2e(\floor{T/2})$, then $T$ is an even spider.
    
    \item If $e(T)=2e(\floor{T/2})+1$, then either $\floor{T/2}=\ceil{T/2}$ or $T$ is a spider with exactly one odd-length leg.
\end{enumerate}
\end{prop}

\begin{proof} Let $v_0$ be the root of $T$. We write $E_1$ for the set of edges incident to $v_0$ but not a leaf and $E_2$ for the set of edges incident to leaves of $T$ but not $v_0$. The key inequality which will enable us to conclude all three properties is
\begin{equation}\label{eq:floorceil}
    e(T)-e(\ceil{T/2})-\abs{E_2}\geq e(\floor{T/2})-\abs{E_1}.
\end{equation}
To prove \eqref{eq:floorceil}, we present an injection $\iota$ from $E(\floor{T/2})$ to $E(T)\setminus E(\ceil{T/2})$ which maps $E_1\setminus E_2$ to $E_2\setminus E_1$. The idea is that $\iota$ should reflect edges in about the midpoint of maximal paths from $v_0$ in $T$ (e.g.\ in the special case where $T$ is a path $v_0\cdots v_\ell$ rooted at $T$, the map $\iota$ should map $v_{i-1}v_i$ to $v_{\ell-i}v_{\ell-i+1}$). We expand on this intuition in the next two paragraphs.

Let $e$ be an edge in $\floor{T/2}$, and let $v$ and $w$ be the endpoints of $e$, chosen so that $w$ is a descendant of $v$. Set $d:=\lambda(w)$ and $\ell:=\lambda_{\max}(w)$, and let $v_0v_1\cdots v_\ell$ be a path from $v_0$ to an arbitrary maximum-level descendant of $w$, so that $v_d=w$ and $v_{d-1}=v$. Set $\iota(e)=v_{\ell-d}v_{\ell-d+1}$. Because $w\in V(\floor{T/2})$, we have $d\leq \ell/2$. In particular,
\[\ell-d+1\geq\ell/2+1>\ceil{\ell/2}\geq d.\]
This implies firstly that $v_{\ell-d+1}$ is a descendant of $w=v_d$, and thus $\lambda_{\max}(v_{\ell-d+1})=\ell$. We conclude that $\ell-d+1>\ceil{\lambda_{\max}(v_{\ell-d+1})/2}$, yielding that $v_{\ell-d+1}\not\in V(\ceil{T/2})$. Therefore the edge $v_{\ell-d}v_{\ell-d+1}$ lies in $E(T)\setminus E(\ceil{T/2})$. Furthermore, if $e\in E_1$ then $d=1$ and $\ell>1$, so the edge $\iota(e)=v_{\ell-1}v_\ell$ is incident to the vertex $v_\ell$ and thus in $E_2$.

We must now argue that $\iota$ is injective, which we do by providing a procedure $\pi$ which reverses the computation of $\iota$. Indeed, given an edge $e':=vw\in E(T)\setminus E(\ceil{T/2})$ with $\lambda(v)<\lambda(w)$, we can recover $\ell=\lambda_{\max}(w)$ and $d=\ell-\lambda(v)$. As $e'\not\in E(\ceil{T/2})$, we have $\ceil{\ell/2}<\ell-d+1$, or equivalently $\floor{\ell/2}\geq d$. Let $v_0v_1\cdots v_\ell$ be a path from $v_0$ to an arbitrary descendant of $w$ of maximum level, and let $\pi(e')=v_{d-1}v_d$. Since $d\leq\floor{\ell/2}$, we have $\pi(e')\in E(\floor{T/2})$. 
Finally, if $e'\in E_2$ then $\ell>1$ and $\lambda(v)=\ell-1$. This means that $d=1$, and so $\pi(e')$ is incident to $v_0$ and thus in $E_1$.  It is clear from the definitions of $\iota$ and $\pi$ that we have $\pi\circ\iota=\operatorname{id}_{E(\floor{T/2})}$. This means that $\iota$ is injective; in particular, the restriction of $\iota$ to $E(\floor{T/2})\setminus E_1$ as a map to $E(T)\setminus(E(\ceil{T/2})\cup E_2)$ is injective. We have thus proven \eqref{eq:floorceil}. 

Now, we derive (1), (2), and (3). Firstly, a direct computation shows that, if $T$ is a spider, then $e(T)-2e(\floor{T/2})$ counts the number of odd-length legs of $T$. Thus (1), (2), and (3) hold when $T$ is a spider. It remains to show that if $T$ is not a spider then $e(T)\geq 2e(\floor{T/2})+1$, and that equality holds only when $\floor{T/2}=\ceil{T/2}$. Both of these properties follow from the claim that
\begin{equation}\label{eq:edge-count}
    \abs{E_2}-\abs{E_1}=\sum_{v\in V(T)\setminus\{v_0\}}\max(\deg (v)-2,0).
\end{equation}
This may be proven by induction on $\abs{V(T)}$. It is true in the base case where $T$ has only a single vertex $v_0$. If $T'$ is formed from $T$ by removing a leaf adjoined to a vertex $w$ which is either $v_0$ or a leaf of $T'$, then each side of \eqref{eq:edge-count} for $T$ is equal to the corresponding expression for $T'$. Otherwise, both sides are greater by exactly one for $T$ as compared to $T'$.

We conclude from \eqref{eq:edge-count} that $\abs{E_2}>\abs{E_1}$ for $T$ a non-spider. Using \eqref{eq:floorceil}, this implies
\[e(T)\geq e(\ceil{T/2})+e(\floor{T/2})+1\geq 2e(\floor{T/2})+1.\]
If equality holds throughout, we must in particular have $e(\ceil{T/2})=e(\floor{T/2})$, and thus $\floor{T/2}=\ceil{T/2}$.
\end{proof}

We remark that \cref{char tree attempt 2} is also provable via induction.

\subsection{Hypercubes and Embeddings}

We now establish some definitions surrounding hypercubes and homomorphisms of trees.

\begin{defn}[Homomorphisms into the hypercube]\label{def:hom-hypercube}
Let $G$ be a properly edge-colored subgraph of the hypercube $Q_\infty$ and let $T$ be a tree with root $v_0$. A homomorphism $\phi\colon T\to G$ is

\begin{defitemize}
    \item \emph{rainbow} if the edges of $T$ are all mapped to edges of different colors in $G$;
    
    \item \emph{distinctly directed on a set of edges} $E\sub E(T)$ if the edges of $\phi(E)$ each have different coordinates; 
    
    \item \emph{distinctly directed} if $\phi$ is distinctly directed on the whole edge set $E(T)$;
    
    \item \emph{doubly distinct} if $\phi$ is both rainbow and distinctly directed; and 
    
    \item \emph{path-distinct} if, for every leaf $v$ of $T$, $\phi$ is distinctly directed on the edge set of the path from $v_0$ to $v$.
\end{defitemize}

We say that a homomorphism $\Phi$ \textit{extends} another homomorphism $\phi$ if the domain of $\Phi$ contains the domain of $\phi$ and if $\Phi$ agrees with $\phi$ on the domain of $\phi$.
\end{defn}

A large portion of our proof is of the following shape: we will start with a subtree $T'$ of the tree $T$ which we seek to embed together with a homomorphism $\phi\colon V(T')\to V(G)$, and we then extend $\phi$ to a rainbow embedding $\Phi\colon V(T)\to V(G)$. The most fundamental example of this is the case when $T$ and $T'$ are paths.

\begin{prop}\label{maya-paths}
Let $P_{n+1}$ denote the path with $n$ edges $v_0v_1\cdots v_n$ and let $n':=\lfl n/2\rfl+1$.  If $G\sub Q_\infty$ is a properly edge-colored graph with $\delta(G)\ge n$ and if $\phi\colon V(P_{n'+1})\to V(G)$ is a doubly distinct homomorphism of $P_{n'+1}=v_0\cdots v_{n'}$, then $\phi$ extends to a rainbow embedding $\Phi$ of $P_{n+1}$.
\end{prop}
\begin{proof}
Let $\phi$ be a doubly distinct homomorphism defined on $v_0, \ldots, v_{n'}$. For each $i > n'$, we iteratively define $\Phi(v_i)$ for our extension $\Phi$ to be any neighbor of $\Phi(v_{i-1})$ such that the edge $\Phi(v_{i-1})\Phi(v_i)$ does not repeat the color of any previously used edge $\Phi(v_j)\Phi(v_{j-1})$ for $2 \le j \le i - 1$, and also does not repeat the coordinate of any edge in $\Phi(E_i)$ where $E_i:=\{v_jv_{j-1}:2i - n - 1 \le j \le i - 1\}$.  Note that $E_i$ is well-defined, in the sense that $2i-n-1\ge 1$ holds since $i\ge n'+1$.  Moreover, we claim that it is always possible to make such a choice of $\Phi(v_i)$.  Indeed, we aim to do this by applying \Cref{fact:extend1} with $x=\Phi(v_{i-1})$, $\xcol$ the set of colors used by the edges $\Phi(v_j)\Phi(v_{j-1})$ for $2 \le j \le i - 1$, and $\xcor$ the set of coordinates used by $\Phi(E_i)$.  Note that the edge $\Phi(v_{i-1})\Phi(v_{i-2})$ (which exists since $i>n'\ge 1$) has both a color in $\xcol$ and a coordinate in $\xcor$, and that
\[\abs{\xcol}+\abs{\xcor}=i-1+(i-1-2i+n+1+1)=n\le \delta(G).\]
Therefore, \Cref{fact:extend1} shows that there exists a suitable choice for $\Phi(v_i)$ as claimed.

By construction, each edge $\Phi(v_{i-1})\Phi(v_i)$ is assigned a distinct color, so $\Phi$ defines a rainbow homomorphism. It remains to show that $\Phi$ is injective.  Suppose, for contradiction, that $\Phi(v_{k}) = \Phi(v_m)$ for some $k< m$. Since $G$ is bipartite, this implies that $m - k$ is even.  We aim to show that the first $(m-k)/2+1$ edges in the walk
\[\Phi(v_{k})\Phi(v_{k+1})\cdots \Phi(v_m)\]
all use distinct coordinates, from which we can conclude, by \Cref{fact:distinct}, that $\Phi(v_k)\ne \Phi(v_m)$.  To this end, consider any integers $j,i$ with $k+1\le j<i\le (m+k)/2+1$.  If $i\le n'$, then by the hypothesis of $\phi$ being doubly distinct, we have that $\Phi(v_i)\Phi(v_{i-1})=\phi(v_i)\phi(v_{i-1})$ uses a coordinate which is distinct from $\Phi(v_j)\Phi(v_{j-1})=\phi(v_j)\phi(v_{j-1})$. Otherwise, if $i>n'$ then $\Phi(v_i)\Phi(v_{i-1})$ does not have the coordinate of any edge in $E_i$.  Note that we have $\Phi(v_j)\Phi(v_{j-1})\in E_i$ since $j\le i-1$ by definition and since
\[2i-n-1\le m+k+2-n-1\le k+1\le j.\]
We conclude that $\Phi(v_i)\Phi(v_{i-1})$ and $\Phi(v_j)\Phi(v_{j-1})$ have distinct coordinates for all $k+1\le j<i \le (m+k)/2+1$, proving that $\Phi(v_k)\ne \Phi(v_m)$ for all $k,m$ by \Cref{fact:distinct}, giving that $\Phi$ is indeed a rainbow embedding. 
\end{proof}

Our main proof relies on breaking up our given tree $T$ into subtrees $T'_1,T'_2,\ldots$, after which we apply the inductive hypothesis to extend partial embeddings $\phi_i$ of each $T'_i$ into embeddings $\Phi_i$ of the entire subtree $T'_i$.  In order to ensure that, say, $\Phi_i$ and $\Phi_j$ do not share any common vertices in their image, we need to impose some extra conditions on the maps $\Phi_i,\Phi_j$ as outlined in the following lemma.

\begin{lem}\label{maya-easy-injectivity}
Suppose $T',T''$ are trees rooted at $v'_0,v''_0$ and let $\Phi'\colon T'\to Q_\infty$ and $\Phi''\colon T''\to Q_\infty$ be homomorphisms such that $\Phi'(v'_0)$ and $\Phi''(v''_0)$ are adjacent in $Q_\infty$.
Suppose additionally that 
\begin{defitemize}
    \item[(a)] $\Phi'$ is path-distinct on $\lcl T'/2\rcl$ and $\Phi''$ is path-distinct on $\lcl T''/2\rcl$,
    \item[(b)] every edge in $\Phi'(\lcl T'/2\rcl)$ has a different coordinate from every edge in $\Phi''(\lcl T''/2\rcl)$, and
    \item[(c)] the coordinate of the edge $\Phi'(v'_0)\Phi''(v''_0)$ differs from the coordinates of all edges in $\Phi'(\lcl T'/2\rcl)$ and $\Phi''(\lcl T''/2\rcl)$.
\end{defitemize}
Then $\Phi'(V(T'))\cap\Phi''(V(T''))=\emptyset$.
\end{lem}

\begin{proof} 
Let $v'_0v'_1\cdots v'_\ell$ and $v''_0v''_1\cdots v''_m$ be any two paths in $T'$ and $T''$ respectively, ending at vertices at levels $\ell\geq 0$ and $m\geq 0$. To show that $\Phi'(V(T'))$ and $\Phi''(V(T''))$ are distinct, it suffices to show that $\Phi'(v'_\ell)\neq\Phi''(v''_m)$ for any such pair of paths.  By hypothesis, the first $\ceil{\ell/2}$ edges of the walk $\Phi'(v'_0)\cdots\Phi'(v'_\ell)$ and the first $\lcl m/2\rcl$ edges of the walk $\Phi''(v''_0)\cdots\Phi''(v''_\ell)$ together with $\Phi'(v_0')\Phi''(v''_0)$ comprise at least $\ceil{\ell/2}+\ceil{m/2}+1>\frac{\ell+m+1}2$ distinct coordinates. Applying \cref{fact:distinct} to this walk gives the result.

\end{proof}

\section{Proof of Main Theorem}\label{sec:main}
Our main technical goal for this section is to prove \Cref{patch-main-thm}, which informally says that any doubly distinct homomorphism $\phi\colon \floor{T/2}\to G$ can be extended to a rainbow embedding $T\to G$.  This, together with the strategy outlined in \cref{remark:dd-embedding} to construct a double distinct homomorphism on at most half the edges of $T$, yields our main result \cref{main theorem}.

As alluded to around \Cref{char tree attempt 2}, the most delicate case for our argument is when $T$ is a spider with at most one odd-length leg.  We handle this special case first via the following result.

\begin{lem}\label{maya-spiders}
Let $S$ be a spider with root $v_0$, let $L_1,\ldots,L_k$ be the legs of $S$, and let $\ell_i:=e(L_i)$ be the length of the $i$th leg for each $1\leq i\leq k$. Suppose $\ell_2,\ldots,\ell_k$ are even, and let $e_1$ be the edge of $L_1$ not in $\lfl S/2\rfl$ which is closest to the root. Then every doubly distinct homomorphism $\phi$ from $\lfl S/2\rfl$ or $\lfl S/2\rfl\cup\{e_1\}$ to a properly edge-colored graph $G\sub Q_\infty$ with $\delta(G)\ge e(T)$ extends to a rainbow embedding $\Phi\colon S\to G$.
\end{lem}

\begin{proof}
We prove the result by induction on $k$ with the $k=1$ case following from \cref{maya-paths}.  In the remainder, suppose $k>1$.   For ease of reading, we divide our remaining proof into three parts: we first propose an algorithm for extending a given doubly distinct homomorphism $\phi$ to a homomorphism $\Phi\colon S\to G$; we then verify that each step of the algorithm is indeed doable; and finally, we verify that the resulting map $\Phi$ is a rainbow embedding. 

It is clearer to specify our embedding by iteratively mapping edges, rather than vertices, of $S$ into $G$; whenever the algorithm defines $\Phi$ on an edge $e=vw$ with $\lambda(v)<\lambda(w)$, it is implicit that $\Phi(v)$ has previously been defined and we are choosing $\Phi(w)$ among the neighbors of $\Phi(v)$ in $G$ to satisfy certain properties.

\medskip\noindent\textbf{Algorithm Overview}.  Let $\phi$ be a doubly distinct homomorphism mapping $\lfl S/2\rfl$ or $\lfl S/2\rfl\cup\{e_1\}$ to $G$.  We outline a three-step algorithm for extending $\phi$ to a map $\Phi\colon S\to G$.
In what follows, let
$S':=\bigcup_{i=2}^kL_i$
denote the sub-spider of $S$ obtained by deleting the leg $L_1$.

\begin{algsteps}
    \step0{
        Initialize $\Phi=\phi$.
        If $\phi$ is not defined on $e_1$, then define $\Phi(e_1)$ so that $\Phi$ is doubly distinct on $\lfl S/2\rfl\cup \{e_1\}$.}

    \step1{Define $\Phi$ on the remainder of $L_1$ so that $\Phi$ is rainbow on $\lfl S/2\rfl\cup L_1$, no edge in $\Phi(L_1)$ has the same coordinate as any edge in $\Phi(\lfl S'/2\rfl)$, and the restriction of $\Phi$ to $L_1$ is injective.}

    \step2{Define $\Phi$ on the remainder of $S'$ so that $\Phi$ is rainbow on $S$ and injective on $S'$.}

\end{algsteps}

\medskip\noindent\textbf{Verifying Each Step is Possible}.  We now show that each edge of $S$ can be embedded as specified in the algorithm.

\begin{algsteps}
    \step0{Suppose $\phi$ is not defined on $e_1$. Write $e_1=vw$ with $\lambda(w)=\lambda(v)+1$. To embed $e_1$, we aim to apply \cref{fact:extend1} to $x=\phi(v)$ with $\xcol$ and $\xcor$ the sets of colors and coordinates, respectively, of the edges in $\phi(E(\lfl S/2\rfl))$. 
   Note that $\abs{\xcol}+\abs{\xcor}=2e(\lfl S/2\rfl)\leq e(S)\leq\delta(G)$ by \cref{char tree attempt 2}. Moreover, there is an edge $vv'\in E(\lfl S/2\rfl)$ incident to $e_1$ (either $\lambda(v)>0$ and $v$ has a neighbor at level $\lambda(v)-1$ in $\lfl L_1/2\rfl$ or $\lambda(v)=0$ and $v=v_0$ has a neighbor in the nonempty spider $\lfl S'/2\rfl$) and its image $\phi(vv')$ has color in $\xcol$ and coordinate in $\xcor$. Applying \cref{fact:extend1} yields a neighbor $y$ of $\phi(v)$ in $G$ such that $\phi(v)y$ has neither color in $\xcol$ or coordinate in $\xcor$, and we define $\Phi(w)=y$.
    }

    \step1{Let $G'$ be the subgraph of $G$ formed by deleting edges whose color or coordinate is used by some edge in $\phi(\lfl S'/2\rfl)$. Again applying \cref{char tree attempt 2}, we have that
        \[\delta(G')\geq\delta(G)-2e(\lfl S'/2\rfl )\geq e(S)-e(S')=e(L_1).\] 
        Let $\phi'\colon (\lfl L_1/2\rfl\cup\{e_1\})\to G$ be the restriction of $\Phi$  to those edges of $L_1$ already mapped; because $\Phi$ has thus far been defined to be rainbow, the image of $\phi'$ lies in $G'$. Applying \cref{maya-paths}, we may extend $\phi'$ to a rainbow embedding $\Phi'\colon L_1\to G'$. Define $\Phi$ on the remainder of $L_1$ to agree with $\Phi'$; noting that the resulting $\Phi$ is thus injective on $L_1$. Additionally, by definition of $G'$, the map $\Phi$ is rainbow on $\lfl S/2\rfl\cup L_1$, and no edge of $\Phi(L_1)\subseteq G'$ has the same coordinate as any edge in $\Phi(\lfl S'/2\rfl)$.
    }
    
    \step2{
        Let $G''$ be the subgraph of $G$ formed by deleting edges whose color is used by some edge in $\Phi(L_1)$. Observe that \[\delta(G'')\geq \delta(G)-e(L_1)\geq e(S)-e(L_1)=e(S').\] Let $\phi''$ be the restriction of $\Phi$ (or equivalently, $\phi$) to $\lfl S'/2\rfl$; its image is contained in $G''$. Applying the inductive hypothesis to the map $\phi''\colon \lfl S'/2\rfl\to G''$, we obtain a rainbow embedding  $\Phi''\colon S'\to G''$ extending $\phi''$. Define $\Phi$ on the remainder of $S'$ to agree with $\Phi''$.  By construction, $\Phi$ is injective on $S'$ and rainbow on $S$.
    }
\end{algsteps}

\medskip\noindent\textbf{Analyzing the Resulting Map}.  Let $\Phi\colon S\to G$ be the homomorphism obtained as a result of this algorithm.  As $\Phi$ is rainbow by construction, it remains only to check that $\Phi$ is an embedding. Because $\Phi$ is injective when restricted to either $S'$ or $L_1$, we need only show that $\Phi(V(L_1)\setminus\{v_0\})$ and $\Phi(V(S'))$ are disjoint sets of vertices.

Let $v_1$ be the child of $v_0$ in $L_1$, and let $L'_1:=S(v_1)$; that is, $L'_1$ is the subtree of $S$ induced by $V(L'_1)=V(L_1)\setminus\{v_0\}$ with root $v_1$.  We aim to apply \cref{maya-easy-injectivity} to the trees $T'=L'_1$ and $T''=S'$ with associated homomorphisms $\Phi'$ and $\Phi''$ given by the restrictions of $\Phi$ to $L'_1$ and $S'$, respectively. First, observe that $E(\lcl L'_1/2\rcl)\cup\{v_0v_1\}=E(\lfl L_1/2\rfl)\cup\{e_1\}$, as both sets comprise the first $\lfl\lambda_{\max}(v_1)/2\rfl+1$ edges of $L_1$, and that $\lcl S'/2\rcl=\lfl S'/2\rcl$. Using these equalities, Step 0 asserts that $\Phi$ is distinctly directed on the edge set
\[E(\lcl L'_1/2\rcl)\cup E(\lcl S'/2\rcl)\cup \{v_0v_1\}=E(\lfl L_1/2\rfl)\cup E(\lfl S'/2\rcl)\cup\{e_1\}=E(\lfl S/2\rfl)\cup\{e_1\}\]
and conditions (a)--(c) of \cref{maya-easy-injectivity} immediately follow. Thus, \cref{maya-easy-injectivity} implies that $\Phi(V(L'_1))$ and $\Phi(V(S'))$ are disjoint, completing the proof that $\Phi$ is an embedding.

\end{proof}

We now prove the main technical result of this paper, in which we show how to extend a partial doubly distinct homomorphism $\phi$ to a full rainbow embedding $\Phi$ provided that the original homomorphism $\phi$ is defined on some carefully chosen set $E\sub E(T)$ of roughly half the edges of $T$.  Our proof follows a delicate inductive argument, and for technical reasons, we require a statement slightly stronger than our ultimate conclusion.  Specifically, we prove by induction that our embedding $\Phi$ can be chosen to avoid a specified ``bad vertex.'' We use this when applying the inductive hypothesis to a subtree of $T$ to ensure that the embedding of the subtree avoids the image of the root $v_0$.

\begin{thm}\label{patch-main-thm}
Let $T$ be a nonempty tree with root $v_0$ and $G$ a properly edge-colored subgraph of the hypercube with $\delta(G)\geq e(T)$. Suppose $\phi\colon\floor{T/2}\to G$ is a doubly distinct homomorphism. Let $\zbad$ be a neighbor of $\phi(v_0)$ in $Q_\infty-E(G)$, with the property that $\phi$ maps no edge of $\floor{T/2}$ to an edge with the same coordinate as edge $\phi(v_0)\zbad$. Then $\phi$ extends to a rainbow embedding of $T$ into $G-\zbad$ which is path-distinct  on $\ceil{T/2}$. 

\end{thm}

\begin{proof}

We prove the result by induction on $\lambda_{\max}(v_0)$, which we recall is the maximum level of a leaf of $T$.  The base case of $\lambda_{\max}(v_0)=1$ corresponds to $T$ being a star, for which the result is trivial.  Suppose that $T$ is some rooted tree with $\lambda_{\max}(v_0)\geq 2$ and that \cref{patch-main-thm} holds for all trees with lesser values of $\lambda_{\max}$ at the root.  

Similarly to \Cref{maya-spiders}, we divide our proof into four parts. First, we introduce necessary notation. Next, we outline an algorithm to extend $\phi$ to a rainbow homomorphism $\Phi\colon T\to G$. We then verify that each step of the algorithm outline is actually possible --- i.e., that each edge of $T$ can be embedded as claimed in the algorithm outline --- with the aid of \cref{fact:extend1,fact:extendr}, \cref{maya-spiders}, and the inductive hypothesis. Lastly,  we prove that the resulting homomorphism $\Phi\colon T\to G$ is a rainbow embedding of $T$ that is path-distinct on $\ceil{T/2}$, relying heavily on \cref{maya-easy-injectivity}.

\medskip\noindent
\textbf{Notation.}
Partition the vertices of $T$ at level 1 (i.e., the children of $v_0$) into three classes. Let $u_1,\ldots,u_{\ell}$ be the leaves of $T$ that are at level 1; let $s_1,\ldots,s_{k}$ be the vertices at level 1 for which $T(s_i)$ is a nonempty spider with even leg lengths; and let $t_1,\ldots,t_m$ be the remaining vertices at level 1 sorted so that, writing $T_i:=T(t_i)$, we have 
\[
1\le e(T_1)-2e(\lfl T_1/2\rfl)\leq e(T_2)-2e(\lfl T_2/2\rfl)\leq\cdots\leq e(T_m)-2e(\lfl T_m/2\rfl).
\]%
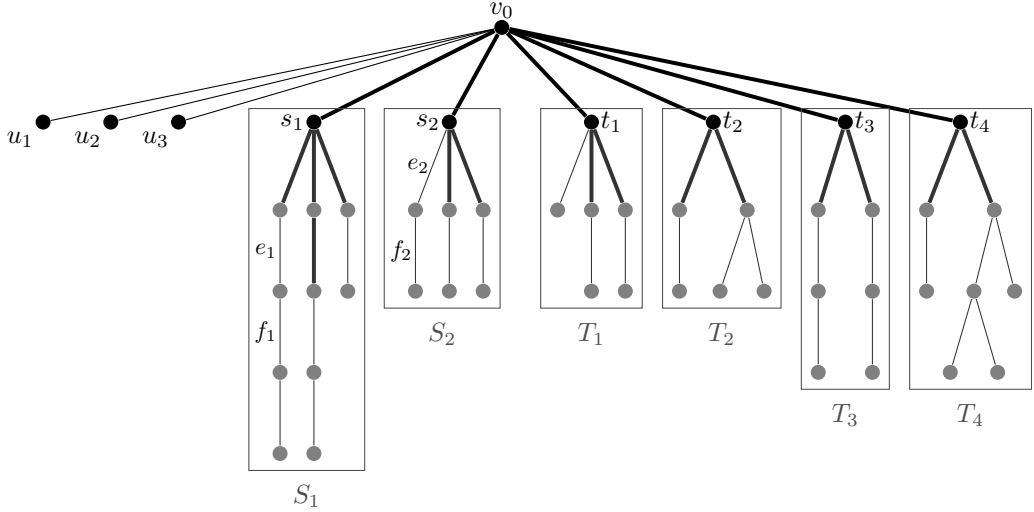
\begin{figure}
    \centering
    
    \tikzset{
        exstyle1/.style={
            scale=0.9,
            baseline=-2.5cm,
            vtx/.style={circle, fill=black, draw=none, inner sep=2pt},
            gvtx/.style={circle, fill=black!50, draw=none, inner sep=2pt},
        }
    }
    
    \begin{tikzpicture}[exstyle1, 
        edgelabel/.style={left=-2pt,pos=0.5,scale=0.9}]
        
        \path (0,0) coordinate[vtx] (u1) node[below left] {$u_1$}
            --++ (1,0) coordinate[vtx] (u2) node[below left] {$u_2$}
            --++ (1,0) coordinate[vtx] (u3) node[below left] {$u_3$}
            --++ (2,0) coordinate[vtx] (s1) node[left] {$s_1$}
            --++ (2,0) coordinate[vtx] (s2) node[left] {$s_2$}
            --++ (2.1,0) coordinate[vtx] (t1) node[right] {$t_1$}
            --++ (1.8,0) coordinate[vtx] (t2) node[right] {$t_2$}
            --++ (1.95,0) coordinate[vtx] (t3) node[right] {$t_3$}
            --++ (1.7,0) coordinate[vtx] (t4) node[right] {$t_4$};
        \path ($(u1)!0.5!(t4)$) --++ (0,1.4) coordinate[vtx] (v0) node[above] {$v_0$};

        \path (s1) --++ (-0.5,-1.3) coordinate[gvtx] (s1a1)
                    --++ (0.5,0) coordinate[gvtx] (s1b1)
                    --++ (0.5,0) coordinate[gvtx] (s1c1)
            (s1a1) --++ (0,-1.2) coordinate[gvtx] (s1a2)
                    -- (s1a2-|s1b1) coordinate[gvtx] (s1b2)
                    -- (s1a2-|s1c1) coordinate[gvtx] (s1c2)
            (s1a2) --++ (0,-1.2) coordinate[gvtx] (s1a3)
                    -- (s1a3-|s1b1) coordinate[gvtx] (s1b3)
            (s1a3) --++ (0,-1.2) coordinate[gvtx] (s1a4)
                    -- (s1a4-|s1b1) coordinate[gvtx] (s1b4);
        \path (s1a1) -- (s1a2) node[edgelabel] {$e_1$}
                    -- (s1a3) node[edgelabel] (s1node) {$f_1$};
        \path (s2) -- (s2|-s1a1) --++ (-0.5,0) coordinate[gvtx] (s2a1)
                    --++ (0.5,0) coordinate[gvtx] (s2b1)
                    --++ (0.5,0) coordinate[gvtx] (s2c1)
            (s2a1|-s1a2) coordinate[gvtx] (s2a2)
                    -- (s2a2-|s2b1) coordinate[gvtx] (s2b2)
                    -- (s2a2-|s2c1) coordinate[gvtx] (s2c2);
        \path (s2) -- (s2a1) node[edgelabel] {$e_2$}
                    -- (s2a2) node[edgelabel] (s2node) {$f_2$};

        \path (t1) -- (t1|-s1a1) --++ (-0.5,0) coordinate[gvtx] (t1a1)
                    --++ (0.5,0) coordinate[gvtx] (t1b1)
                    --++ (0.5,0) coordinate[gvtx] (t1c1)
            (t1a1|-s1a2) coordinate (t1a2)
                    -- (t1a2-|t1b1) coordinate[gvtx] (t1b2)
                    -- (t1a2-|t1c1) coordinate[gvtx] (t1c2)
            (t1a1|-s1a3) coordinate (t1a3);

        \path (t2) --++ (-0.5,0|-s1a1) coordinate[gvtx] (t2a1)
                    --++ (1,0) coordinate[gvtx] (t2b)
                    -- (t2b|-s1a2) --+ (-0.4,0) coordinate[gvtx] (t2ba)
                                    --+ (0.25,0) coordinate[gvtx] (t2bb)
                    -- (t2a1|-s1a2) coordinate[gvtx] (t2a2)
                    ;

        \path (t3) -- (t3|-s1a1) --+ (-0.4,0) coordinate[gvtx] (t3a1)
                    --+ (0.4,0) coordinate[gvtx] (t3b1)
            (t3a1|-s1a2) coordinate[gvtx] (t3a2)
            (t3b1|-s1a2) coordinate[gvtx] (t3b2)
            (t3a1|-s1a3) coordinate[gvtx] (t3a3)
            (t3b1|-s1a3) coordinate[gvtx] (t3b3);
        
        \path (t4) --++ (-0.5,0|-s1a1) coordinate[gvtx] (t4a1)
                    --++ (1,0) coordinate[gvtx] (t4b)
                    -- (t4a1|-s1a2) coordinate[gvtx] (t4a2)
                    -- (t4b|-s1a2) --+ (-0.3,0) coordinate[gvtx] (t4ba)
                                    --+ (0.3,0) coordinate[gvtx] (t4bb)
                    -- (t4ba|-s1a3) --+ (-0.35,0) coordinate[gvtx] (t4baa)
                                    --+ (0.35,0) coordinate[gvtx] (t4bab);
                    ;

        \foreach \x in {u1,u2,u3} {\draw (v0) -- (\x);}
        \foreach \x in {s1,s2,t1,t2,t3,t4} {\draw[line width=1.5pt] (v0) -- (\x); }
        
        \foreach \x/\y in 
        {
         s1/{s1a1,s1b1,s1c1},   s1a1/s1a2, s1a2/s1a3, s1a3/s1a4,
                                s1b1/s1b2, s1b2/s1b3, s1b3/s1b4,
                                s1c1/s1c2,
         s2/{s2a1,s2b1,s2c1},   s2a1/s2a2, s2b1/s2b2, s2c1/s2c2,
         t1/{t1a1,t1b1,t1c1},   t1b1/t1b2, t1c1/t1c2,
         t2/{t2a1,t2b}, t2a1/t2a2,  t2b/{t2ba,t2bb},
         t3/{t3a1,t3b1},t3a1/t3a2,t3a2/t3a3,t3b1/t3b2,t3b2/t3b3,
         t4/{t4a1,t4b}, t4a1/t4a2,  t4b/{t4ba,t4bb},    t4ba/{t4baa,t4bab}%
         } {
            \foreach \z in \y {
                \draw[black!80] (\x) -- (\z);
            }
        }
        \foreach \x/\y in 
        {
         s1/{s1a1,s1b1,s1c1}, s1b1/s1b2,
         s2/{s2b1,s2c1}, 
         t1/{t1b1,t1c1},
         t2/{t2a1,t2b},
         t3/{t3a1,t3b1},
         t4/{t4a1,t4b}%
         } {
            \foreach \z in \y {
                \draw[line width=1.5pt, black!80] (\x) -- (\z);
            }
        }

        \begin{scope}[black!70]
            \foreach \a/\b/\c/\d/\lab in {
                s1/s1node/s1a4/s1c1/S_1, s2/s2node/s2a2/s2c2/S_2,
                t1/t1a1/t1c2/t1c2/T_1, t2/t2a1/t2bb/t2bb/T_2, 
                t3/t3a1/t3b3/t3b1/T_3, t4/t4a2/t4bab/t4bb/T_4
            } {
                \draw ($(\a-|\b) + (-0.25,0.2)$) rectangle ($(\c-|\d) + (0.25,-0.25)$);
                \path (\c-|\b) -- (\c-|\d) node[pos=0.5, below=0.3] {$\lab$};
            }
        \end{scope}
    \end{tikzpicture}
    
    \caption{A tree with vertices, edges, and subtrees labeled according to the notation used in the proof of \cref{patch-main-thm}. In bold are the edges of $\A\cup\B$.
    }
    \label{fig:notation}
\end{figure}%

Here we emphasize that $1\leq e(T_i)-2e(\lfl T_i/2\rfl)$ holds for each $1\leq i\leq m$ by \Cref{char tree attempt 2}: if $e(T(v))=2e(\lfl T(v)/2\rfl)$, then either $v$ is a leaf of $T$ or $T(v)$ is a nonempty spider with even leg lengths. 

For $1\leq i\leq k$, set $S_i:=T(s_i)$ and let $L_i$ be an arbitrary leg of $S_i$. Let $\lambda_i$ be half the length of $L_i$, and enumerate the vertices of $L_i$ as $s_i,s^{(i)}_1,\ldots,s^{(i)}_{2\lambda_i}$ in order. 
Let $e_i$ and $f_i$ be the edges $e_i:=s^{(i)}_{\lambda_i-1}s^{(i)}_{\lambda_i}$ and $f_i:=s^{(i)}_{\lambda_i}s^{(i)}_{\lambda_i+1}$. 
Note that because $S_i$ is nonempty, this path has length at least 2, and hence both $e_i$ and $f_i$ exist.

For $1\leq i\leq k$, let $A_i:=\{v_0s_i\}\cup E(\lfl S_i/2\rfl)\setminus\{e_i\}$. For $1\leq j\leq m$, let $B_j:=\{v_0t_j\}\cup E(\lfl T_j/2\rfl)$. We notate
\[\A:=\bigcup_{i=1}^kA_i
\quad\text{and}\quad
\B:=\bigcup_{j=1}^mB_j.\]
For convenience, define $\A_{\geq i}:=\bigcup_{i'=i}^kA_i$ and $\A_{>i}:=\bigcup_{i'=i+1}^kA_i$, and define $\B_{\geq j}$ and $\B_{>j}$ similarly.
We shall frequently use the observations that 
\begin{equation}\abs{A_i}=e(S_i)/2
    \quad\text{and}\quad
    2\abs{B_j}=2+2e(\lfl T_j/2\rfl)\leq 1+e(T_j),\label{eq:AB}\end{equation}
which together imply
\begin{equation}\label{eq:ABUnion}
    \abs*{\A\cup\B}=\abs*{\bigcup_{i=1}^kA_i\cup\bigcup_{j=1}^mB_j}\le \sum_{i=1}^k\frac{e(S_i)}2+\sum_{j=1}^m\frac{e(T_j)+1}2=\frac {e(T)-k-\ell}2.
\end{equation}

Let $\coox$ be the coordinate of the edge $\phi(v_0)\zbad$ in $Q_\infty-E(G)$.

\medskip
Having fully introduced our notation, we may outline an algorithm for constructing the desired rainbow extension of $\phi$.

\medskip\noindent
\textbf{Algorithm Overview.} Let $\phi\colon \floor{T/2}\to G$ be a doubly distinct homomorphism as described in the hypothesis of the theorem. We outline an algorithm to extend $\phi$ to a rainbow homomorphism $\Phi\colon T\to G$. As in the proof of \cref{maya-spiders}, we specify $\Phi$ by iteratively mapping edges, rather than vertices, of $T$ into $G$.
For an illustration of the algorithm, we refer the reader to \Cref{fig:steps}, which depicts a tree with every edge marked with the step at which it is embedded.

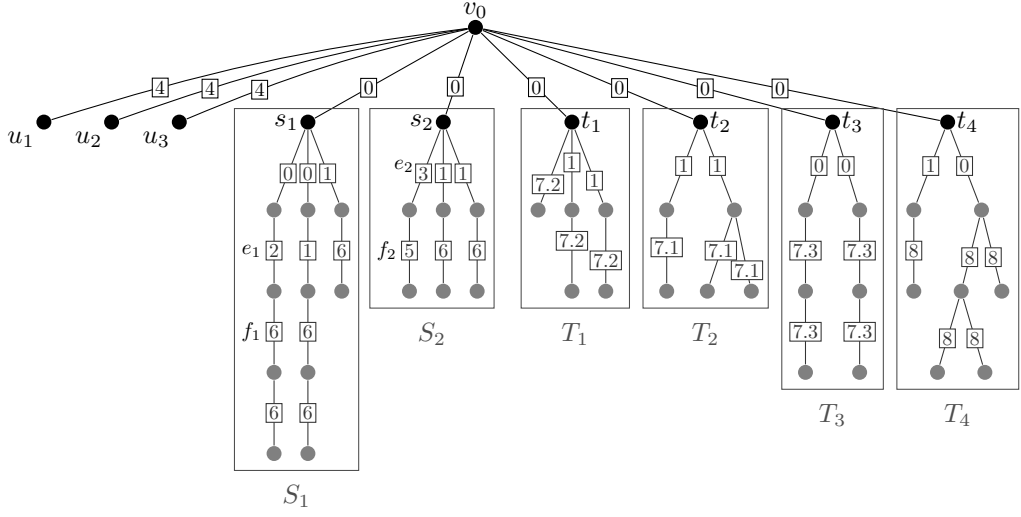
\begin{figure}
    \centering
    
    \tikzset{
        exstyle1/.style={
            scale=0.9,
            baseline=-2.5cm,
            vtx/.style={circle, fill=black, draw=none, inner sep=2pt},
            gvtx/.style={circle, fill=black!50, draw=none, inner sep=2pt},
        }
    }
    
    \begin{tikzpicture}[exstyle1, scale=1,
        edgelabel/.style={pos=0.5,scale=0.75, draw=black!80, inner sep=1.5pt, fill=white},
        eflabel/.style={pos=0.5,left=2pt,scale=0.8}]
        
        \path (0,0) coordinate[vtx] (u1) node[below left] {$u_1$}
            --++ (1,0) coordinate[vtx] (u2) node[below left] {$u_2$}
            --++ (1,0) coordinate[vtx] (u3) node[below left] {$u_3$}
            --++ (1.9,0) coordinate[vtx] (s1) node[left] {$s_1$}
            --++ (2.0,0) coordinate[vtx] (s2) node[left] {$s_2$}
            --++ (1.9,0) coordinate[vtx] (t1) node[right] {$t_1$}
            --++ (1.9,0) coordinate[vtx] (t2) node[right] {$t_2$}
            --++ (1.95,0) coordinate[vtx] (t3) node[right] {$t_3$}
            --++ (1.7,0) coordinate[vtx] (t4) node[right] {$t_4$};
        \path ($(u1)!0.5!(t4)$) --++ (-0.3,1.4) coordinate[vtx] (v0) node[above] {$v_0$};

        \path (s1) --++ (-0.5,-1.3) coordinate[gvtx] (s1a1)
                    --++ (0.5,0) coordinate[gvtx] (s1b1)
                    --++ (0.5,0) coordinate[gvtx] (s1c1)
            (s1a1) --++ (0,-1.2) coordinate[gvtx] (s1a2)
                    -- (s1a2-|s1b1) coordinate[gvtx] (s1b2)
                    -- (s1a2-|s1c1) coordinate[gvtx] (s1c2)
            (s1a2) --++ (0,-1.2) coordinate[gvtx] (s1a3)
                    -- (s1a3-|s1b1) coordinate[gvtx] (s1b3)
            (s1a3) --++ (0,-1.2) coordinate[gvtx] (s1a4)
                    -- (s1a4-|s1b1) coordinate[gvtx] (s1b4);
        \path (s1a1) -- (s1a2) node[eflabel] {$e_1$}
                    -- (s1a3) node[eflabel] (s1node) {$f_1$};
        \path (s2) -- (s2|-s1a1) --++ (-0.5,0) coordinate[gvtx] (s2a1)
                    --++ (0.5,0) coordinate[gvtx] (s2b1)
                    --++ (0.5,0) coordinate[gvtx] (s2c1)
            (s2a1|-s1a2) coordinate[gvtx] (s2a2)
                    -- (s2a2-|s2b1) coordinate[gvtx] (s2b2)
                    -- (s2a2-|s2c1) coordinate[gvtx] (s2c2);
        \path (s2) -- (s2a1) node[eflabel,left] {$e_2$}
                    -- (s2a2) node[eflabel] (s2node) {$f_2$};

        \path (t1) -- (t1|-s1a1) --++ (-0.5,0) coordinate[gvtx] (t1a1)
                    --++ (0.5,0) coordinate[gvtx] (t1b1)
                    --++ (0.5,0) coordinate[gvtx] (t1c1)
            (t1a1|-s1a2) coordinate (t1a2)
                    -- (t1a2-|t1b1) coordinate[gvtx] (t1b2)
                    -- (t1a2-|t1c1) coordinate[gvtx] (t1c2)
            (t1a1|-s1a3) coordinate (t1a3);

        \path (t2) --++ (-0.5,0|-s1a1) coordinate[gvtx] (t2a1)
                    --++ (1,0) coordinate[gvtx] (t2b)
                    -- (t2b|-s1a2) --+ (-0.4,0) coordinate[gvtx] (t2ba)
                                    --+ (0.25,0) coordinate[gvtx] (t2bb)
                    -- (t2a1|-s1a2) coordinate[gvtx] (t2a2)
                    ;

        \path (t3) -- (t3|-s1a1) --+ (-0.4,0) coordinate[gvtx] (t3a1)
                    --+ (0.4,0) coordinate[gvtx] (t3b1)
            (t3a1|-s1a2) coordinate[gvtx] (t3a2)
            (t3b1|-s1a2) coordinate[gvtx] (t3b2)
            (t3a1|-s1a3) coordinate[gvtx] (t3a3)
            (t3b1|-s1a3) coordinate[gvtx] (t3b3);
        
        \path (t4) --++ (-0.5,0|-s1a1) coordinate[gvtx] (t4a1)
                    --++ (1,0) coordinate[gvtx] (t4b)
                    -- (t4a1|-s1a2) coordinate[gvtx] (t4a2)
                    -- (t4b|-s1a2) --+ (-0.3,0) coordinate[gvtx] (t4ba)
                                    --+ (0.3,0) coordinate[gvtx] (t4bb)
                    -- (t4ba|-s1a3) --+ (-0.35,0) coordinate[gvtx] (t4baa)
                                    --+ (0.35,0) coordinate[gvtx] (t4bab);
                    ;

        \foreach \x/\y/\i/\pos in 
        {
         s1/{s1a1,s1b1}/0/0.6, s1/s1c1/1/0.6,
            s1a1/s1a2/2/0.5, s1a2/s1a3/6/0.5, s1a3/s1a4/6/0.5,
            s1b1/s1b2/1/0.5, s1b2/s1b3/6/0.5, s1b3/s1b4/6/0.5,
            s1c1/s1c2/6/0.5,
         s2/s2a1/3/0.6, s2/{s2b1,s2c1}/1/0.6, 
            s2a1/s2a2/5/0.5, s2b1/s2b2/6/0.5, s2c1/s2c2/6/0.5,
         t1/t1a1/7.2/0.75, t1/t1b1/1/0.45,t1/t1c1/1/0.7,
            t1b1/t1b2/7.2/0.35, t1c1/t1c2/7.2/0.65,
         t2/{t2a1,t2b}/1/0.5, t2a1/t2a2/7.1/0.45,
            t2b/t2ba/7.1/0.5, t2b/t2bb/7.1/0.8,
         t3/{t3a1,t3b1}/0/0.5, 
            t3a1/t3a2/7.3/0.5, t3a2/t3a3/7.3/0.5, 
            t3b1/t3b2/7.3/0.5, t3b2/t3b3/7.3/0.5,
         t4/t4a1/1/0.5,t4/t4b/0/0.5, t4a1/t4a2/8/0.5,  t4b/{t4ba,t4bb}/8/0.6,    t4ba/{t4baa,t4bab}/8/0.6%
         } {
            \foreach \z in \y {
                \draw[black!80] (\x) -- (\z) node[edgelabel, pos=\pos] (\x\z) {\i};
            }
        }

        
        \foreach \x in {s1,s2,t1,t2,t3,t4} {
            \draw (v0) -- (\x) node[edgelabel, pos=0.65, draw=black] {0};
        }
        \foreach \x in {u1,u2,u3} {
            \draw (v0) to[bend right=5]  node[edgelabel, pos=0.75, draw=black] {4} (\x);
        }

        \path (t1c1) --++ (0.1,0) coordinate (t1right)
            (t2a1) --++ (-0.1,0) coordinate (t2left)
            (t3a1) --++ (-0.1,0) coordinate (t3left)
            (t3b1) --++ (0.1,0) coordinate (t3right);

        \begin{scope}[black!70]
            \foreach \a/\b/\c/\d/\lab in {
                s1/s1node/s1a4/s1c1/S_1, s2/s2node/s2a2/s2c2/S_2,
                t1/t1a1/t1c2/t1right/T_1, t2/t2left/t2bb/t2bb/T_2, 
                t3/t3left/t3b3/t3right/T_3, t4/t4a2/t4bab/t4bb/T_4
            } {
                \draw ($(\a-|\b) + (-0.25,0.2)$) rectangle ($(\c-|\d) + (0.25,-0.25)$);
                \path (\c-|\b) -- (\c-|\d) node[pos=0.5, below=0.3] {$\lab$};
            }
        \end{scope}
    \end{tikzpicture}
    
    \caption{The same tree as \Cref{fig:notation}, with each edge labeled by the step in which it is embedded. Edges present in $\lfl T/2\rfl$ are marked with 0, to indicate that their embedding is initially specified in $\phi$. The symbols 7.1--7.3 correspond to Cases 1--3 of the proof that Step 7 is possible.
    }
    \label{fig:steps}
\end{figure}%

\begin{algsteps}
    
    \step1{
        Initialize $\Phi=\phi$. Define $\Phi$ on the remainder of $\A\cup\B$ so that $\Phi$ is doubly distinct on $\A\cup\B$ and additionally so that no edge in $\Phi(\A\cup\B)$ has coordinate $\coox$. 
    }

    \step2{
        If $k\geq 1$, define $\Phi(e_1)$ to be an edge with a distinct color and coordinate from each edge in $\Phi(\A\cup\B)$ and additionally a distinct coordinate from $\coox$.
    }
    
    \step3{
        If $k\geq 2$, define $\Phi(e_2),\ldots,\Phi(e_k)$ to have pairwise distinct colors, such that for each $2\leq i\leq k$, the edge $\Phi(e_i)$ also avoids the colors and coordinates of $\Phi(\A\cup\B\cup\{e_1\})$.
    }

    \step4{
        Define $\Phi(v_0u_1),\ldots,\Phi(v_0u_\ell)$ in that order such that, for each $1\leq i\leq\ell$, the edge $\Phi(v_0u_i)$ avoids the colors and coordinates of  $\Phi(\A\cup\B\cup\{e_1,\ldots,e_k\})\cup\{v_0u_1,\ldots,v_0u_{i-1}\}$.
    }

    \step5{
        Define $\Phi(f_2),\ldots,\Phi(f_k)$ in that order such that $\Phi(f_i)$ is a distinct color from every edge in
        \[\Phi\left(\A\cup\B\cup\{e_1,\ldots,e_k\}\cup\{v_0u_1,\ldots,v_0u_\ell\}\cup\{f_2,\ldots,f_{i-1}\}\right),\]
        and moreover $\Phi(f_i)$ has a distinct coordinate from every edge in $\Phi\left(\A_{\geq i}\cup\B\cup \{e_i\}\right)$.
    }
    
    \step6{
        Define $\Phi$ on the remainder of $\bigcup_{i=1}^{k}E(S_i)$ so that $\Phi$ is injective on each $S_i$ and, moreover, the edges $\Phi\left(\bigcup_{i=1}^kE(S_i)\cup\{v_0s_1,\ldots,v_0s_k\}\cup\{v_0u_1,\ldots,v_0u_\ell\}\cup \B\right)$ are all different colors. (Note that this is exactly the set of edges embedded thus far).
    }
    
    \step7{
        For $j=1,\ldots,m-1$, define $\Phi$ on the remainder of $E(T_j)$ so that $\Phi$ is injective on $T_j$, so that its restriction to $\lcl T_j/2\rcl$ is path-distinct, and so that no edge of $\Phi(\ceil{T_j/2})$ has the same coordinate as any edge in $\Phi(\{v_0t_j\}\cup \B_{>j})$. Additionally, we require that $\Phi$ never maps any two edges to the same color, i.e., that at the end of step 7 the resulting mapping of $T-(T_m-\lfl T_m/2\rfl)$ is rainbow.
    }
    
    \step8{
        If $m\geq 1$, define $\Phi$ on the remainder of $E(T_m)$ so that $\Phi$ is injective on $\{v_0t_m\}\cup T_m$ and so that the entire resulting homomorphism $\Phi\colon T\to G$ is rainbow.
    }
\end{algsteps}

\medskip

We now verify that each step of the proposed algorithm is indeed possible.

\medskip\noindent
\textbf{Each Step Is Possible.} We show that each edge of $T$ can be embedded as described in the algorithm. To simplify our notation, we say a set of edges $E\subseteq E(T)$ \emph{uses} a given color or coordinate at a given point in the algorithm if $\Phi(e)$ has that color or coordinate for some $e\in E(T)$ on which $\Phi(e)$ has already been defined.

We repeatedly apply \cref{fact:extendr,fact:extend1} in a manner akin to the proof of \cref{maya-spiders}. More formally, suppose we wish to define $\Phi$ on an edge $e=vw$ with $w$ a descendant of $v$; that is, we wish to define $\Phi(w)$ to be a neighbor of $\Phi(v)$ in $G$ such that the edge $\Phi(v)\Phi(w)$ avoids certain sets $\xcol$ and $\xcor$ of forbidden colors and coordinates, respectively. In this setting, we implicitly apply \cref{fact:extendr} or \ref{fact:extend1} to the vertex $x=\Phi(v)$, obtaining a neighbor $y$ of $x$ which is a suitable choice for $\Phi(w)$. To show that there are $r$ edges $xx'\in E(G)$ with color in $\xcol$ and coordinate in $\xcor$, we exhibit $r$ edges $vv'\in T$ whose images under $\Phi$ use distinct colors from $\xcol$ and distinct coordinates from $\xcor$.

\begin{proofsteps}
    \step1{
    Let $e'_1,e'_2,\ldots$ be an arbitrary ordering of the edges in $\A\cup\B\sm E(\floor{T/2})$. For $h=1,2,\ldots$, we define $\Phi(e'_h)$ using \cref{fact:extendr} as described above to avoid the set $\xcol$ of colors used by $E(\lfl T/2\rfl)\cup\{e'_1,\ldots,e'_{h-1}\}$ and the set $\xcor$ comprising of $\coox$ and the coordinates used by $E(\lfl T/2\rfl)\cup\{e'_1,\ldots,e'_{h-1}\}$. We have $\abs{\xcol}+\abs{\xcor}\leq \abs{\A\cup\B}-1+\abs{\A\cup\B}\leq e(T)-1<\delta(G)$, so we may apply the $r=0$ case of \cref{fact:extendr} to define $\Phi(e'_h)$.
    }

    \step2{
        If $k\ge 1$ (which is assumed to hold whenever Step 2 occurs), then $\A\cup\B$ has size at most $\frac 12(e(T)-1)$ by \eqref{eq:ABUnion}. We use \cref{fact:extend1} to define $\Phi(e_1)$ to avoid the set $\xcol$ of colors used by $\A\cup\B$ and the set $\xcor$ comprising $\coox$ and the coordinates used by $\A\cup\B$. We have $\abs{\xcol}+\abs{\xcor}= 2\abs{\A\cup\B}+1\leq e(T)\leq\delta(G)$. Additionally, $e_1$ is incident to some edge of $A_1$, and any such edge uses a color in $\xcol$ and coordinate in $\xcor$. Thus, \cref{fact:extend1} yields a valid choice of $\Phi(e_1)$.
    }

    \step3{
        To embed $e_i$ for $i\geq 2$, apply \cref{fact:extend1} to the set $\xcol$ of colors used by $\A\cup\B\cup\{e_1,\ldots,e_{i-1}\}$ and the set $\xcor$ of coordinates used by $\A\cup\B\cup\{e_1\}$. We have 
        \[\abs{\xcol}+\abs{\xcor}=\abs{\A\cup\B}+(i-1)+\abs{\A\cup\B}+1
        \leq 2\abs{\A\cup\B}+k\leq e(T)\leq\delta(G).
        \]
        Additionally, $e_i$ is incident to some edge of $A_i$, and any such edge uses a color in $\xcol$ and coordinate in $\xcor$, so the hypotheses of \cref{fact:extend1} are satisfied.
    }

    \step4{
        Fix $1\leq i\leq\ell$ and let $\xcol$ and $\xcor$ be the sets of colors and coordinates used by edges in $\A\cup\B\cup\{e_1,\ldots,e_k\}\cup\{v_0u_1,\ldots,v_0u_{i-1}\}$. To embed $\Phi(v_0u_i)$, we require the full generality of \cref{fact:extendr}. By Step 1 and previous iterations of Step 4, the $r=k+i-1$ edges $v_0s_1,\ldots,v_0s_k,v_0u_1,\ldots,v_0u_{i-1}$ use $r$ distinct colors in $\xcol$ and $r$ distinct coordinates in $\xcor$. Applying \eqref{eq:ABUnion} yields
        \[
        \abs{\xcol}+\abs{\xcor}-r\leq 2(\abs{\A\cup\B}+k+i-1)-(k+i-1)
        \leq e(T)-\ell+i-1\leq e(T)-1<\del(G),
        \]
        so the hypothesis of \cref{fact:extendr} is satisfied.
    }

    \step5{
        Suppose $k\geq 2$. To embed $f_i$ for $2\leq i\leq k$, apply \cref{fact:extend1} to the set $\xcol$ of colors used by $\A\cup\B\cup\{e_1,\ldots,e_k\}\cup\{v_0u_1,\ldots,v_0u_\ell\}\cup\{f_2,\ldots,f_{i-1}\}$ and the set $\xcor$ of coordinates used by $\A_{\geq i}\cup\B\cup\{e_i\}$. We have
        {\begin{align*}
        \abs{\xcol}+\abs{\xcor}
        &=\left(\abs{\A}+\abs{\B}+k+\ell+i-2\right)
            +\left(\abs{\A_{\geq i}}+\abs{\B}+1\right)
        \\&=\sum_{i'=1}^{i-1}(\abs{A_{i'}}+2)
            +\sum_{i'=i}^k(2\abs{A_{i'}}+1) +2\abs{\B}+\ell
        \\&\leq\left(\sum_{i'=1}^k(2\abs{A_{i'}}+1)\right)
            +2\abs{\B}+\ell
        =2(\abs{\A\cup\B})+k+\ell\leq e(T)\leq\delta(G),
        \end{align*}}
        where the penultimate inequality uses \eqref{eq:ABUnion}. Additionally, $f_i$ is incident to $e_i$ which uses a color in $\xcol$ and coordinate in $\xcor$, so the hypothesis of \cref{fact:extend1} is satisfied.
    }

    \step6{
        We iteratively define $\Phi$ on the remainder of $E(S_i)$ for $i=1,\ldots,m$. At the $i$th step, let $G'_i$ be the subgraph of $G$ formed by deleting edges of any color already used by some edge in $E(T)\setminus E(S_i)$, noting that 
        \[\delta(G'_i)\geq\delta(G)-e(T)+e(S_i)\geq e(S_i).\] Let $\phi'_i$ be the restriction of $\Phi$ to the edges of $S_i$ that have already been mapped; that is, the domain of $\phi'_i$ is $E(\lfl S_i/2\rfl)$ if $i=1$ and $E(\lfl S_i/2\rfl)\cup\{f_i\}$ if $i\geq 2$. By \cref{maya-spiders}, we may extend $\phi'_i$ to a rainbow embedding $\Phi'_i\colon S_i\to G'_i$. Define $\Phi$ to agree with $\Phi'_i$ on the remainder of $S_i$. Because $\Phi'_i$ is rainbow on $S_i$ and $G'_i$ avoids the colors of all previously mapped edges in $E(T)\setminus E(S_i)$, the map $\Phi$ remains a rainbow homomorphism.
    }

    \step7{
        We handle the analysis of this step depending on whether $e(T_j)-2e(\lfl T_j/2\rfl)=1$ and whether $\lcl T_j/2\rcl=\lfl T_j/2\rfl$,. Sub-trees falling into each case are pictured in \Cref{fig:steps}.   Our proofs primarily rely on applying the inductive hypothesis of the theorem to relevant subtrees of $T$.  By doing this, we  additionally show that $\Phi(T_j)$ is disjoint from $\Phi(v_0)$ in some of our cases, though we emphasize that this is not explicitly required of us to establish the validity of Step 7.

        \medskip\noindent
        Case 1: Suppose $\lcl T_j/2\rcl=\lfl T_j/2\rfl$.
        Step 1 asserts that $\Phi$ is doubly distinct on $\lcl T_j/2\rcl\cup\{v_0t_j\}\cup\B_{>j}=\B_{\geq j}$. In particular, $\Phi$ has already been defined to be path-distinct on $\lcl T_j/2\rcl$ and no edge of $\lcl T_j/2\rcl$ uses the same coordinate as any edge of $\{v_0t_j\}\cup\B_{>j}$. It remains to define $\Phi$ on the remainder of $T_j$ without reusing any color so that the restriction of $\Phi$ to $T_j$ is injective.
        
        Let $G''_j$ be the subgraph of $G$ formed by deleting edges of any color already used by some edge in $E(T)\setminus E(T_j)$ and let $\phi''_j\colon \lfl T_j/2\rfl\to G''_j$ be the restriction of $\Phi$ to those edges of $T_j$ mapped in a prior step. We have $\delta(G''_j)\geq\delta(G)-e(T)+e(T_j)\geq e(T_j)$ and $\Phi(v_0t_j)\notin E(G''_j)$. Moreover, because $\Phi$ is doubly distinct on $\lcl T_j/2\rcl\cup\{v_0t_j\}$, the map $\phi''_j$ is doubly distinct and no edge of $\phi''_j(\lcl T_j/2\rcl)$ has the same coordinate as $\Phi(v_0t_j)$. Thus, the inductive hypothesis applies to extend $\phi''_j$ to a rainbow embedding $\Phi''_j\colon T_j\to G''_j-\Phi(v_0)$. Define $\Phi$ to agree with $\Phi''_j$ on the remainder of $T_j$; the definition of $G''_j$ implies that $\Phi$ reuses no color.

        \medskip\noindent
        Case 2: Suppose $\lcl T_j/2\rcl\neq\lfl T_j/2\rfl$ and $e(T_j)-2e(\lfl T_j/2\rfl)=1$. In this case, $T_j$ is a spider with exactly one leg of odd length by \Cref{char tree attempt 2}. Let $f'_j$ be the middle edge of the odd-length leg, i.e., the unique element of $E(\lcl T_j/2\rcl) \setminus E(\lfl T_j/2\rfl)$. First, we define $\Phi(f'_j)$ to avoid both the set $\xcol$ of all colors previously used as well as the set $\xcor$ of coordinates used by $\B_{\geq j}$. For each $j'\geq j$, exactly $e(T_{j'})-e(\lfl T_{j'}/2\rfl)=e(T_{j'})+1-\abs{B_{j'}}$ edges of $T_{j'}$ remain unembedded, so by \eqref{eq:AB},
        {\[\abs{\xcol}
        =e(T)-\sum_{j'=j}^m\left(e(T_{j'})+1-\abs{B_{j'}}\right)
        \leq e(T)-\sum_{j'=j}^m\abs{B_{j'}}
        =e(T)-\abs{\B_{\geq j'}}
        \leq\delta(G)-\abs{\xcor}.
        \]}
        Moreover, $f'_j$ is incident to some edge of $B_j$, and any such edge uses a color in $\xcol$ and coordinate in $\xcor$. Thus, \cref{fact:extend1} allows us to define $\Phi(f'_j)$ as desired. 
       The map $\Phi$ is doubly distinct on $\B_{\geq j}\cup\{f'_j\}=E(\lcl T_j/2\rcl)\cup\{v_0t_j\}\cup\B_{>j}$, so $\Phi$ is path-distinct on $\lcl T_j/2\rcl$ and no edge of $\lcl T_j/2\rcl$ uses the same coordinate as any edge of $\{v_0t_j\}\cup\B_{>j}$. It remains to define $\Phi$ on the remainder of $T_j$ without reusing any color so that the restriction of $\Phi$ to $T_j$ is injective.

        We follow a similar proof to Case 1. Let $G''_j$ be the subgraph of $G$ formed by deleting edges of any color already used by some edge in $E(T)\setminus E(T_j)$ and let $\phi''_j\colon (\lfl T_j/2\rfl\cup\{f'_j\})\to G''_j$ be the restriction of $\Phi$ to those edges of $T_j$ already mapped. We have $\delta(G''_j)\geq\delta(G)-e(T)+e(T_j)\geq e(T_j)$; furthermore, $\phi''_j$ is doubly distinct on $\lfl T_j/2\rfl$. By \cref{maya-spiders}, the map $\phi''_j$ extends to a rainbow embedding $\Phi''_j\to G''_j$. Define $\Phi$ to agree with $\Phi''_j$ on the remainder of $T_j$; as before, $\Phi$ reuses no color by definition of $G''_j$.

        \medskip\noindent
        Case 3: Suppose $e(T_j)-2e(\lfl T_j/2\rfl)\geq 2$. Let $G''_j$ be the subgraph of $G$ formed by deleting edges of any color already used by some edge in $E(T)\setminus E(T_j)$ or any coordinate that is used by some edge in $\{v_0v_j\}\cup\B_{>j}$. Let $\phi''_j$ be the restriction of $\Phi$ to $\lfl T_j/2\rfl$. Because $\Phi$ never reuses colors and is doubly distinct on $\B_{\geq j}=E(\lfl T_j/2\rfl)\cup\{v_0v_j\}\cup\B_{>j}$, the restriction $\phi''_j$ is a doubly distinct map from $\lfl T_j/2\rfl$ to $G''_j$ that never uses the coordinate of $\Phi(v_0t_j)$. Moreover, $\Phi(v_0t_j)\notin E(G''_j)$ by definition. To apply the inductive hypothesis, it remains to lower-bound $\delta(G''_j)$.
        
        For each $j'>j$, exactly $e(T_{j'})-e(\lfl T_{j'}/2\rfl)$ edges of $T_{j'}$ remain unembedded, and the ordering of $T_1,\ldots,T_m$ implies that
        \[e(T_{j'})-2e(T_{j'})\geq e(T_j)-2e(T_j)\geq 2.\]
        Thus, the number of colors already used by some edge in $E(T)\setminus E(T_j)$ is at most
        {\begin{align*}
        e(T)-e(T_j)-\sum_{j'=j+1}^m(e(T_{j'})-e(\lfl T_{j'}/2\rfl)
        &\leq e(T)-e(T_j)-\sum_{j'=j+1}^m(e(\lfl T_{j'}/2\rfl)+2)
        \\&=e(T)-e(T_j)-\abs{\B_{>j}}-(m-j)\leq e(T)-e(T_j)-\abs{\B_{>j}}-1.
        \end{align*}}
        It follows that $G''_j$ is obtained by deleting edges from at most $(e(T)-e(T_j)-\abs{B_{>j}}-1)+(\abs{B_{>j}}+1)=e(T)-e(T_j)$ classes of colors or coordinates from $G$, so $\delta(G''_j)\geq\delta(G)-(e(T)-e(T_j))\geq e(T_j)$.

        We can thus apply the inductive hypothesis to extend $\phi''_j\colon \lfl T_j/2\rfl\to G''_j$ to a rainbow embedding $\Phi''_j\colon T_j\to G''_j-v_0$ which is path-distinct on $\lcl T_j/2\rcl$. Define $\Phi$ to agree with $\Phi''_j$ on the remainder of $T_j$. By definition of $G''_j$, the resulting $\Phi$ reuses no color and additionally, no edge of $T_j$ uses the same coordinate as any edge in $\{v_0t_j\}\cup\B_{>j}$.
    }
    
    \step8{
        Suppose $m\geq 1$. Let $G''_m$ be the subgraph of $G$ formed by deleting edges of any color used by some edge in $E(T)\setminus E(T_m)$, so $\del(G''_m)\geq\del(G)-e(T)+e(T_m)\geq e(T_m)$. Let $\phi''_m\colon \lfl T_m/2\rfl\to G''_m$ be the restriction of $\Phi$ to those edges of $T_m$ already mapped. Because $\Phi$ is doubly distinct on $B_m=\{v_0t_m\}\cup E(\lfl T_m/2\rfl)$, the restriction $\phi''_m$ is doubly distinct and never uses the coordinate of $\Phi(v_0t_m)$. Moreover, $\Phi(v_0t_m)\notin E(G''_m)$ by definition. Apply the inductive hypothesis to extend $\phi''_m$ to a rainbow embedding $\Phi''_m\to G''_m-v_0$ and define $\Phi$ to agree with $\Phi''_m$ on the remainder of $T_m$. The resulting $\Phi$ is injective on $\{v_0t_m\}\cup T_m$ because $\Phi''_m$ maps $T_m$ injectively to $G''_m-v_0$, and $\Phi$ reuses no color by definition of $G''_m$.
    }
\end{proofsteps}

\medskip\noindent 
\textbf{The Algorithm Terminates in a Rainbow Path-Distinct  Embedding Avoiding $\zbad$.} Let $\Phi\colon T\to G$ be the result of this algorithm. We check each of our required conditions in turn.

\medskip 

\noindent \textit{The map $\Phi$ is a rainbow homomorphism}.  Indeed, every step of our algorithm specifies that the map $\Phi$ constructed is rainbow at the end of that step. So, the fact that $\Phi$ is rainbow has been verified in the previous section.

\medskip

\noindent \textit{The map $\Phi$ is path-distinct on $\ceil{T/2}$}.   It suffices to show the stronger statement that, for each $v\in V(T)$ at level 1, $\Phi$ maps $\{v_0v\}\cup E(\lfl T(v)/2\rfl)$ to edges with distinct coordinates. If $v=s_i$, then this edge set is $A_i\cup\{e_i\}$; we imposed in Steps 1 and 3 that these edges be mapped to edges with different coordinates. If $v=t_j$ then this edge set is $B_j$; we imposed in Step 1 that these edges be mapped in different coordinates. Finally, if $v=u_i$ then this set has size $1$ and there is nothing to show.

\medskip

\noindent \textit{The map $\Phi$ is injective}. Recall that $\Phi$ restricts to an injective map on each of the subtree $S_1,\ldots,S_k,T_1,\ldots,T_m$ by Steps 6, 7, and 8. It remains to show that the images of the substructures (i.e., singleton vertices or subtrees) $v_0,u_1,\ldots,u_\ell,S_1,\ldots,S_k,T_1,\ldots,T_m$ are vertex-disjoint under $\Phi$. 

The remainder of the proof largely follows by repeated applications of \cref{maya-easy-injectivity} to pairs of subtrees $(T',T'')$ of $T$, with the corresponding homomorphisms $(\Phi',\Phi'')$ always assumed to be the restrictions of $\Phi$ to $(T',T'')$. When checking \cref{maya-easy-injectivity}(a), we often use the fact that $\Phi$ is path-distinct on $\ceil{T/2}$ (and thus on any subtree thereof as well).

First we show that, for any $1\leq i\leq\ell$, the map $\Phi$ does not identify $u_i$ with any vertex of $T-u_i$. We apply \cref{maya-easy-injectivity} to $T'=T-u_i$ and $T''=\{u_i\}$. Condition (a) follows because $\Phi$ is path-distinct on $\ceil{T/2}\supseteq\lcl T'/2\rcl$ and condition (b) is trivial. Condition (c) requires that no edge in $\Phi(\lcl T'/2\rcl)$ has the same coordinate as $\Phi(v_0u_i)$. We have $E(\lcl T'/2\rcl)=\A\cup\B\cup\{e_1,\ldots,e_k\}\cup\{v_0u_{i'}:i'\neq i\}$, so this follows from Step 4.
We conclude by \cref{maya-easy-injectivity} that $\Phi$ does not identify $u_i$ with any vertex of $T-u_i$.

Next, we show that $\Phi$ is injective on $\{v_0\}\cup T_1\cup\cdots\cup T_m$. Fix $1\leq j_1<j_2\leq m$; we apply \cref{maya-easy-injectivity} to $T'=T_{j_1}$ and $T''=\{v_0t_{j_2}\}\cup T_{j_2}$. Condition (a) follows because $\Phi$ is path-distinct on $\lcl T'/2\rcl$ by Step 7 and is path-distinct on $\ceil{T/2}\supseteq\lcl T''/2\rcl$. 
For (b) and (c), Step 7 asserts that no edge of $E(\lcl T'/2\rcl)$ receives the same coordinate under $\Phi$ as any edge in $\{v_0t_{j_1}\}\cup E(\lcl T''/2\rcl)=\{v_0t_{j_1}\}\cup B_{j_2}$, and moreover Step 1 asserts that $\Phi$ is distinctly directed on $\{v_0t_{j_1}\}\cup E(\lcl T''/2\rcl)\subseteq\B$. By \cref{maya-easy-injectivity}, it follows that $\Phi$ does not identify any vertex of $T_{j_1}$ with any vertex of $T_{j_2}$ for any $1\leq j_1<j_2\leq m$; moreover, taking $j_2=m$, we see that $\Phi$ does not identify any vertex of $T_{j_1}$ with $v_0$ for $j_1<m$. As $\Phi$ is injective on $\{v_0t_m\}\cup T_m$ by Step 8, we conclude that $\Phi$ is injective on $\{v_0\}\cup T_1\cup\cdots\cup T_m$.

Next, we show that $\Phi$ identifies no vertex of $S_i$ (if $i=1$) or $S_i-s^{(i)}_{2\lambda_i}$ (if $i\geq 2$) with any vertex of $T-s_i$ (which we recall has vertex set $V(T)-V(S_i)$). Fix $1\leq i\leq k$. We apply \cref{maya-easy-injectivity} to $T'=S_i$ (if $i=1$) or $T'=S_i-s^{(i)}_{2\lambda_i}$ (if $i\geq 2$) and $T''=T-s_i$. Condition (a) holds as $\Phi$ is distinctly directed on $E(\lcl T'/2\rcl)\subseteq A_i\cup\{e_1\}$ and path-distinct on $\ceil{T/2}\supseteq\lcl T''/2\rcl$. For (b) and (c), observe that
\[E(\lcl T'/2\rcl)\cup\{v_0s_i\}\subseteq\A\cup\{e_1\}\quad\text{ and }\quad E(\lcl T''/2\rcl)\subseteq \A\cup \B\cup\{e_1,\ldots,e_k\}\cup\{v_0u_1,\ldots,v_0u_\ell\}.\]
Steps 1--4 assert that no edge $e'\in \A\cup\{e_1\}$ receives the same coordinate under $\Phi$ as any $e''\in\A\cup\B\cup\{e_1,\ldots,e_k\}\cup\{v_0u_1,\ldots,v_0u_\ell\}$ if $e'\neq e''$. Conditions (b) and (c) follow, so \cref{maya-easy-injectivity} implies that $\Phi$ identifies no vertex of $S_i$ (if $i=1$) or $S_i-s_{2\lambda_i}^{(i)}$ (if $i\geq 2$) with any vertex of $T-s_i$.

Thus far, we have shown that $\Phi$ identifies each leaf $u_i$ with no other vertex of $T$, that $\Phi$ is injective on $\{v_0\}\cup T_1\cup\cdots\cup T_m$, and that $\Phi$ identifies no vertex of $S_i$ (if $i=1$) or $S_i-s^{(i)}_{2\lambda_i}$ (if $i\geq 2$) with any vertex of $T-s_i$. To complete the proof that $\Phi$ is injective, we need only show that $\Phi$ never identifies any two vertices in the set $\{s^{(i)}_{2\lambda_i}:2\leq i\leq k\}$.

Fix indices $2\leq h<i\leq k$. To show that $\Phi(s^{(i)}_{2\lambda_i})\neq\Phi(s^{(h)}_{2\lambda_h})$, we apply \cref{fact:distinct} to the walk
\[
\Phi(s^{(i)}_{2\lambda_i})\cdots \Phi(s^{(i)}_1)\Phi(s_i)\Phi(v_0)\Phi(s_h)\Phi(s^{(h)}_1)\cdots\Phi( s^{(h)}_{2\lambda_h})\]
in $Q_\infty$, which has length $2\lambda_i+2\lambda_h+2$. Note that all $\lambda_i+\lambda_h+2$ edges of the path 
$$s^{(i)}_{\lambda_i-1}\cdots s^{(i)}_1s_iv_0s_hs^{(h)}_1\cdots s^{(h)}_{\lambda_h+1}$$
are contained in the set $A_i\cup A_h\cup\{e_h,f_h\}$; thus, by Steps 1, 3, and 5, their images under $\Phi$ have distinct coordinates.
It follows by \cref{fact:distinct} that $\Phi$ cannot identify two vertices of the form $s^{(h)}_{2\lambda_h}$ and $s^{(i)}_{2\lambda_i}$, completing the proof that $\Phi$ is injective.

\medskip
\noindent \textit{The homomorphism $\Phi$ maps no vertex to $\zbad$}.  
We first show that most vertices of $T$ do not map to $\zbad$ by applying \cref{maya-easy-injectivity}. The lemma is applied to the tree $T'=T-\{s^{(i)}_{2\lambda_i}:2\leq i\leq k\}$ formed by removing the leaves $s_{2\lambda_2}^{(2)},\ldots,s_{2\lambda_k}^{(k)}$ from $T$, and to a rooted tree $T''$ with one vertex. Let $\Phi'\colon T'\to Q_\infty$ be the restriction of $\Phi$ to $T'$, and let $\Phi''\colon T''\to Q_\infty$ map the single vertex to $\zbad$. Condition (a) of \cref{maya-easy-injectivity} follows because $\Phi$ is path-distinct on $\ceil{T/2}\supseteq\lcl T'/2\rcl$ and condition (b) is trivial because $E(T'')=\emptyset$. Condition (c) requires that no edge of
\[E(\lcl T'/2\rcl)=\A\cup\B\cup\{e_1\}\cup\{v_0u_1,\ldots,v_0u_\ell\}\]
 uses the coordinate $\coox$, which follows from Steps 1, 2, and 4. Thus, the hypotheses of \cref{maya-easy-injectivity} are satisfied, and we conclude that no vertex of $T'$ is mapped to $\zbad$ by $\Phi$.

It remains to show that the vertices of $V(T)\setminus V(T')=\{s_{2\lambda_2}^{(2)},\ldots,s_{2\lambda_k}^{(k)}\}$ are not mapped to $\zbad$. Fix $2\leq i\leq k$ and consider the path $v_0s_is^{(i)}_1\cdots s^{(i)}_{\lambda_i+1}$ of length $\lambda_i+2$, all of whose edges are contained in $A_i\cup\{e_i,f_i\}$.
These edges are mapped to edges with distinct coordinates by Steps 1, 3, and 5.
Thus, applying \cref{fact:distinct} to the walk $\zbad\Phi(v_0)\Phi(s_i)\Phi(s^{(i)}_1)\cdots\Phi(s^{(i)}_{2\lambda_i})$ 
in $Q_\infty$, which has length $2\lambda_i+2$, we conclude that $\zbad\neq\Phi(s^{(i)}_{2\lambda_i})$.
Hence, $\Phi$ maps no vertex of $T$ to $\zbad$.

\end{proof}

Finally, we use \Cref{patch-main-thm} to conclude our main result.

\begin{proof}[Proof of \Cref{main theorem}]
Let $G$ be a properly edge-colored subgraph of $Q_\infty$ and $T$ a tree with $e(T)\leq\delta(G)$. We construct a rainbow embedding $\Phi:T\to G$.

By repeatedly applying \cref{fact:extendr}, it is straightforward to construct a doubly distinct homomorphism $\phi:\lfl T/2\rfl\to G$. Indeed, order the edges of $\lfl T/2\rfl$ as $e_1,e_2,\ldots$ so that $\{e_1,\ldots,e_i\}$ induces a connected subgraph of $T$ for each $1\leq i\leq e(\lfl T/2\rfl)$. Define $\phi(e_1)$ arbitrarily and for $i=2,3,\ldots$, iteratively define $\phi(e_i)$ to avoid the sets $\xcol$ and $\xcor$ of colors and coordinates, respectively, used by any edge in $\phi(\{e_1,\ldots,e_{i-1}\})$. To do so we apply \cref{fact:extendr} as in the proof of \cref{patch-main-thm}; the hypothesis of \cref{fact:extendr} is satisfied for $r=0$ because
\[
\abs{\xcol}+\abs{\xcor}=2(i-1)\leq 2e(\lfl T/2)-2\leq e(T)-2<\delta(G),
\]
with the center inequality using \cref{char tree attempt 2}.
Applying \Cref{patch-main-thm} to the resulting $\phi$ (and choosing $\zbad$ arbitrarily among vertices of $Q_\infty$ satisfying the necessary condition) gives a rainbow embedding of $T$ into $G$, proving the result. 
\end{proof}

\section{Concluding Remarks}\label{sec:con}
In this paper we resolved \Cref{main conjecture} by showing that every proper edge-coloring of $Q_n$ contains a rainbow copy of every tree on at most $n$ edges.  This result, together with Schrijver's Conjecture, naturally motivates the following broad question: given a graph $G$, what trees $T$ are such that every proper edge-coloring of $G$ contains a rainbow copy of $T$?  Perhaps the most natural special case of this question is the following, which can be viewed as a vast generalization of Schrijver's Conjecture.

\begin{quest}\label{general hosts question}
Does there exist some $C>0$ such that for every graph $G$ of minimum degree $d$ and every tree $T$ on at most $d-C$ edges, every proper edge-coloring of $G$ contains a rainbow copy of $T$?
\end{quest}

We can say slightly more to \cref{general hosts question} than what is given by \cref{main theorem}: while \Cref{main theorem} is stated only in terms of subgraphs of the hypercube, the exact same proof can be made to work for a slightly broader class of graphs.  Specifically, we say that a graph $G$ is \textit{matching-separable} if for every edge $xy\in E(G)$, there exists a matching $M_{xy}$ such that $x$ and $y$ lie in distinct connected components of $G-M_{xy}$.  We note that our current proof can be extended to give the following.
\begin{thm}
If $G$ is a matching-separable graph of minimum degree at least $d$, then every proper edge-coloring of $G$ contains a copy of every rainbow tree $T$ on at most $d$ edges.
\end{thm}
Note that every subgraph of the hypercube is matching-separable, since if $xy\in E(G)\sub E(Q_n)$ is an edge whose vertices differ on the $i$th coordinate, then taking $M_{xy}$ to be the set of all edges of $G$ whose vertices differ on the $i$th coordinate gives the desired property.  Thus this result in particular recovers \Cref{main theorem}.  Moreover, there exist matching-separable graphs which are not subgraphs of a hypercube.  Examples of such graphs include odd cycles of length at least 5 (via taking $M_e$ to consist of the two edges incident to $e$), as well as the Petersen graph (via choosing $M_e$ such that $G-M_e$ is the disjoint union of two 5-cycles).
\begin{proof}[Sketch of Proof]
We construct our rainbow embedding $\Phi\colon V(T)\to V(G)$ exactly like we do in the proof of \Cref{main theorem}, except now when choosing the image of some edge $e\in V(T)$, instead of making sure $\Phi(e)$ does not use the same coordinate of some set of previously mapped to edges $\{\Phi(e_1),\ldots,\Phi(e_t)\}$, we instead make sure that $\Phi(e)$ does not lie in any of the matchings $M_{\Phi(e_i)}$.  

Observe that in order for a homomorphism $\Phi$ of $T$ to satisfy $\Phi(u)=\Phi(v)$, it must be the case that for each edge $e'$ in the path from $u$ to $v$ in $T$, some other edge $e''$ in this path satisfies $\Phi(e'')\in M_{\Phi(e')}$, as otherwise $\Phi(u),\Phi(v)$ would be in distinct connected components of $G=M_{\Phi(e')}$ (and hence in particular not be equal to each other).  As such, to guarantee $\Phi(u)\ne \Phi(v)$ it suffices for more than half the edges $e'$ in the path from $u$ to $v$ to satisfy that no other edge in the path maps into $M_{\Phi(e')}$, with this condition being exactly analogous to our proofs of injectivity for the hypercube which relied on \Cref{fact:distinct}.  Moreover, each $M_{\Phi(e')}$ set forbids at most one neighbor for any given vertex of $G$ because $M_{\Phi(e')}$ is a matching, so the exact same proof of \Cref{fact:extendr} goes through if we replace the condititon of $G$ being a subgraph of a hypercube to being matching-separable.
\end{proof}

Another direction, in the spirit of \Cref{directed Schrijver}, is to consider directed graphs.  In particular, our proof can be used to show that any properly edge-colored directed graph $D$ whose underlying graph is a subgraph of $Q_n$ and which has minimum out-degree and in-degree at least $d$ contains a rainbow copy of every directed tree on at most $d$ edges. It is possible to weaken these degree conditions somewhat depending on the structure of $T$. Specifically, if $T$ is a directed tree such that every vertex has in-degree at most 1 (which is equivalent to saying $T$ has some ``root'' vertex and every arc is directed away from the root), then only minimum out-degree at least $d$ is needed for our present proof to work.  It is not clear to us what degree conditions should be necessary in order to guarantee that directed subgraph of $Q_n$ contain rainbow copies of a given tree $T$.

Lastly, it is natural to ask if one can replace the condition of \Cref{main theorem} that $G$ has minimum degree at least $d$ with the condition that it has \textit{average} degree at least $d$.  Indeed, this is the main conjecture that was made in \cite{cks-25} and serves as a rainbow analog of the classic Erd\H{o}s--S\'os Conjecture.  We do not currently know how to adapt our proof to the average degree setting even in relatively simple cases.  In particular, even the following case of this problem remains open for $d\ge 5$.

\begin{conj}
If $G$ is a subgraph of a hypercube $Q_n$ and if $G$ has average degree at least $d$, then every proper edge-coloring of $G$ contains a rainbow copy of the path on $d$ edges.
\end{conj}

\subsection*{Acknowledgments}

This material is based upon work supported by the National Science Foundation under Grant No.~DMS-1928930, while the authors were in residence at the Simons Laufer Mathematical Sciences Institute in Berkeley, California, during the semester of Spring 2025.  Crawford is partially supported by National Science Foundation Grant DMS-2152498. Sankar is supported by a Fannie and John Hertz Foundation Fellowship and a National Science Foundation Graduate Research Fellowship Program under Grant No.~DGE-1656518.  Schildkraut is supported by a National Science Foundation Graduate Research Fellowship Program under Grant No.~DGE-2146755.  Spiro is supported by a National Science Foundation Mathematical Sciences Postdoctoral Research Fellowship under Grant No.~DMS-2202730.

\bibliography{literature}
\bibliographystyle{yuval}

\end{document}